\documentclass[aos,preprint]{imsart}

\RequirePackage{amsthm,amsmath,amsfonts,amssymb, stmaryrd, bbm}
\RequirePackage[numbers]{natbib}
\RequirePackage{graphicx}

\startlocaldefs
\theoremstyle{plain}
\newtheorem{corollary}{Corollary}
\newtheorem{theorem}{Theorem}
\newtheorem{lemma}{Lemma}
\newtheorem{proposition}{Proposition}

\theoremstyle{definition}

\newtheorem{definition}{Definition}
\newtheorem{remark}{Remark}

\DeclareMathOperator{\Id}{Id}

\DeclareMathOperator{\Fix}{Fix}
\DeclareMathOperator{\Der}{Der}

\DeclareMathOperator{\diag}{diag}
\DeclareMathOperator{\Dom}{Dom}
\DeclareMathOperator{\Ker}{Core}

\DeclareMathOperator{\Neg}{Neg}

\DeclareMathOperator{\Unif}{Unif}

\newcommand{\1}{\mathbbm{1}}
\newcommand{\ds}{\displaystyle}
\newcommand{\dsum}{\ds\sum}
\newcommand{\eps}{\varepsilon}

\newcommand{\llb}{\llbracket}
\newcommand{\rrb}{\rrbracket}

\newcommand{\E}{\mathbb{E}}
\newcommand{\N}{\mathbb{N}}
\newcommand{\R}{\mathbb{R}}
\newcommand{\mc}{\mathcal}
\newcommand{\Proba}{\mathbb{P}}
\newcommand{\Sym}{\mathcal{S}}

\newcommand{\defeq}{\overset{\text{def}}{=}}

\newcommand{\goesto}[1]{\underset{#1}{\longrightarrow}}
\endlocaldefs

\begin{document}

\begin{frontmatter}
\title{Graph alignment in sparse inhomogeneous models via self-overlap}
\runtitle{Graph alignment in sparse inhomogeneous models via self-overlap}

\begin{aug}
\author[A]{\fnms{Louis}~\snm{Vassaux}\ead[label=e1]{louis.vassaux@inria.fr}}

\address[A]{INRIA Paris\printead[presep={,\ }]{e1}}
\end{aug}

\begin{abstract}
We develop a general framework for understanding when graph alignment is information-theoretically feasible in sparse inhomogeneous random graph models, by studying the set of vertices on which the underlying matching can be recovered. Our main theorem gives a general lower bound on this set by leveraging the balanced load function introduced in \cite{hajek1990}. The corresponding obstruction is captured by a new graph parameter, the self-overlap, which measures the extent to which a graph can imitate itself under a non-trivial relabelling. We then show that this criterion is sharp in a broad class of sparse inhomogeneous models, recovering known Erdős--Rényi phenomena and yielding sharp thresholds for Chung--Lu graphs and stochastic block models.

\end{abstract}



\end{frontmatter}


\section{Introduction}

The \textbf{graph alignment problem} is a well-studied statistical problem which aims to answer the following question: given two correlated but anonymously labelled graphs, when can one recover a non-trivial fraction of the hidden vertex correspondence? In this work, we build a model-agnostic framework to determine when graph alignment is information-theoretically feasible.

\subsection{Alignable subsets}

We will study the graph alignment problem from the following perspective.
 Given two correlated graphs $(G, H)$ over $n$ vertices, what is the largest subset of vertices on which we are capable of determining the correct matching? 

This approach is inspired by the approach in \cite{du2025}, in which the above question is answered in the Erd\H{o}s-R\'enyi case. In particular, the author of \cite{du2025} proves that in this case, the maximal alignable subset is (up to some specific ambiguities) equal to the subset
\begin{equation} \label{benchmark}
    V_{> \rho} = \{ v \in [n], w_I(v) > \rho\}
\end{equation}
where:

\begin{itemize}
    \item $I = G \land_{\sigma^*} H = \sigma^*(G) \cap H$ is the \textbf{true intersection} graph, with regards to the planted matching $\sigma^*$;
    \item $w_I$ is the \textbf{balanced load function} associated with $I$, as defined in \cite{hajek1990} (we will recall its definition in Section \ref{mainresults});
    \item $\rho$ is some model-dependent parameter.
\end{itemize}

This result highlights the relevance of the balanced load function in the study of graph alignment; the map $w_I$ emerges as a benchmark to test whether a given vertex $v$ can be correctly mapped or not. We will push this result to a much larger class of graphs, evidencing the fact that the appearance of the balanced load function is not just a quirk of the Erd\H{o}s-R\'enyi case.

Our key tool here is the introduction of the \textbf{self-overlap}, a scalar $\rho \in [0, +\infty]$ associated with any sequence of large graphs measuring how self-symmetric said graphs are. In a certain sense, this quantity measures how well a graph can imitate itself under a non-trivial relabelling of its vertices. It will become the benchmark in (\ref{benchmark}): vertices whose balanced load in the true intersection graph lies above the self-overlap are alignable, whereas under certain conditions vertices below this threshold cannot be distinguished due to the self-symmetries of the underlying graph.

Thus, rather than viewing graph alignment feasibility as a purely model-specific question, we isolate two separate ingredients. The first is the amount of information carried by each vertex in the true intersection graph, as measured by the balanced load function. The second is the intrinsic self-similarity of the underlying graph sequence, as measured by the self-overlap. The main results of this paper show that the comparison between these two quantities governs partial and almost exact graph alignment in a broad class of sparse random graph models.

\subsection{Relation to previous work}

Most of the theoretical literature on graph alignment focuses on specific graph models, the two most prominent examples being the correlated Erd\H{o}s-Rényi model \cite{dingdu2023, dingduli2025, ganassalimassoulie2020, wuxuyu2022} and the correlated Gaussian model \cite{fanmaowuxu2023, ganassali2022gaussian, massoulievarmavassauxwaldspurger2026,  wuxuyu2022}. These models are fairly computationally tractable, which has allowed for many precise results to be proven; in particular, both settings now come with near-complete phase diagrams, both for the feasibility/infeasibility and computationally easy/hard phase transitions.

More recently, there has been a growing interest in generalising beyond these simple models. One research avenue has been to generalise the models in various directions, for instance by adding more graphs \cite{newah1, ameenhajek2026detection, evenganassali2025, vassauxmassoulie2026}, by adding asymmetry between the two graphs \cite{maiermassoulie2026}, or by adding node features \cite{yarandiganassali2026}. Another has been to prove results in other random graph models which more accurately represent real-world data, including Chung--Lu graphs \cite{yuxulin2021} and geometric random graphs \cite{evenganassalimaiermassoulie2024, wangwuxuyolou2022}.
In this paper, we take a slightly different route and try to build a framework which is as model-agnostic as possible, in the same spirit as \cite{raczsridhar2023}. 

Here, our focus will be on \textbf{partial} and \textbf{almost exact} graph alignment, rather than \textbf{exact} graph alignment. In other words, we regard two permutations as essentially equivalent if they agree on all but $o(n)$ vertices. Both formulations of the problem have been studied extensively, but they are naturally adapted to different regimes. Exact alignment is most relevant in relatively dense settings, where the average degree tends to infinity, while partial or almost exact guarantees are often the right notion in sparse graphs, where the average degree remains of order one.

Our main contribution is to shift the focus from global model parameters to the structure of the alignable set itself. In sparse inhomogeneous graphs, some vertices may be intrinsically much easier to match than others, and the relevant question is therefore not only whether alignment is possible, but where it is possible. The framework developed in this paper gives a general way to identify these vertices and to quantify the obstructions which prevent the remaining ones from being matched, in a form that applies across a broad class of sparse random graph models.

We will also provide evidence that our results are sharp in a fairly broad range of models, including sparse Chung--Lu graphs and sparse stochastic block models. This suggests that the balanced-load/self-overlap criterion is not merely a convenient sufficient condition, but captures a genuine obstruction to graph alignment across a variety of sparse random graph models.

\subsection*{Organisation}
The rest of the paper is organised as follows. In Section \ref{model}, we introduce the graph alignment model and the class of correlated inhomogeneous graph systems considered in the paper. In Section \ref{mainresults}, we state our main results, after recalling the balanced load function and defining the self-overlap. Section \ref{corollaries} derives consequences for partial and almost exact alignment in several standard sparse random graph models. Section \ref{balancedload} collects the basic properties of the balanced load function, while Section \ref{sov} discusses the self-overlap and its interpretation. Sections \ref{proofersov} and \ref{proofimpossible} contain the proofs of the main sharpness and impossibility results.
\subsection*{Notation}
In this paper, $n$ is a large integer. We set $V = V_n = [n]$ and $E = \{ \{ u, v \}: u \neq v \in V \}$. If $e = \{u, v \} \in E$, and $\sigma \in \Sym_n$, we denote $\sigma(e) = \{ \sigma(u), \sigma(v) \} \in E$. 

If $K$ is a graph over $V$ then $|K|$ is the number of edges of $K$; we will sometimes identify $K$ with the set of its edges. $\vec E(K)$ is the set of directed edges in $K$. If $V_1, V_2 \subseteq V$, then $K[V_1 \leftrightarrow V_2]$ is the graph over $V$ whose edge set is \[\{e \in K \text{ with one endpoint in } V_1 \text{ and the other in } V_2 \}.\] In particular, $K[V_1] \defeq K[V_1 \leftrightarrow V_1]$ is the induced subgraph on $V_1$. 

If $K, K'$ are graphs over $V$, we will say that $K \simeq_{n \to +\infty} K'$ if the symmetric difference of the two graphs verifies $|K \Delta K'| = o(n)$. We will say that $K \lesssim_{n \to +\infty} K'$ if there exists $K'' \subseteq K'$ such that $K \simeq_{n \to +\infty} K''$. For any $\sigma \in \Sym_n$, $K \land_\sigma K'$ is the graph $\sigma(K) \cap K'$; $K \lor_\sigma K'$ is the graph $\sigma(K)  \cup K'$. 

If $p = (p_n) \geq 0$, the law $\mc{G}(n, p)$ is the standard Erd\H{o}s-R\'enyi law over graphs on $V$.

If $\sigma, \sigma' \in \Sym_n$, we define their overlap $ov(\sigma, \sigma') = \frac{1}{n} \# \{ v \in [n], \sigma(v) = \sigma'(v) \}$. The distance $d = 1-ov$ makes $\Sym_n$ into a metric space; any (closed) balls $B(\pi, r)$ for $\pi \in \Sym_n, r \geq 0$ are defined with regards to this metric. We also denote $\Fix \sigma = \{ v \in V: \sigma(v) = v \}$ and $\Der \sigma = V \setminus \Fix \sigma$.

If $(X_n)_{n \geq 0}$ is a sequence of real-valued random variables, we define
\begin{equation}
    \operatorname*{plimsup}_{n \to +\infty} X_n = \sup \left\{ x \in \R : \limsup_{n \to +\infty}\Proba(X_n \geq x) > 0 \right\}
\end{equation}
as well as $\ds\operatorname*{pliminf}_{n \to +\infty} X_n = -\operatorname*{plimsup}_{n \to +\infty} (-X_n)$. 

If $(Y_n)$ is another sequence of real-valued random variables, we say that $X_n = o_\Proba(Y_n)$ (resp. $X_n = O_\Proba(Y_n)$) if there exists a random sequence $(c_n)$ such that $X_n = c_n Y_n$ and $\operatorname*{plimsup}_{n \to +\infty} |c_n| = 0$ (resp. $(c_n)$ is tight). We say that $X_n = \Theta_\Proba(Y_n)$ if $X_n = O_\Proba(Y_n)$ and $Y_n = O_\Proba(X_n)$.
\section{The model} \label{model}
The graph alignment problem, in full generality, can be stated formally as follows.

\begin{definition} \label{graphalignment}
    Let $(K_1, K_2)$ be a pair of random graphs, following a joint law $\mu$. The graph alignment problem is the following: a permutation $\sigma^* \in \Sym_n$ is chosen uniformly at random, independently from $K_1, K_2$; from the observation of the graphs $(G, H) = ((\sigma^*)^{-1}(K_1), K_2)$, we wish to recover the matching $\sigma^*$.

    In this setting, we refer to $(G, H)$ as a \textbf{correlated graph system}.
\end{definition}

The correlated graph systems which we study here will typically consist of inhomogeneous random graphs, which are defined as follows.

\begin{definition}
    Let $(p_e)_{e \in E}$ be a random vector of elements of $[0,1]$. An \textbf{inhomogeneous random graph} (sampled from $(p_e)$) is a graph $K$ such that, conditional on $(p_e)_{e\in E}$, the variables $(\1_{e\in K})_{e\in E}$ are independent and satisfy
\begin{equation}
    \1_{e\in K} \sim \operatorname{Ber}(p_e).
\end{equation}

If, for $\mathbf{e} \in E$ chosen uniformly at random, the sequence of random variables $(n p_{\mathbf{e}})_{n \geq 1}$ is uniformly integrable, we say that $K$ is \textbf{weakly inhomogeneous}.
\end{definition}

This is a fairly standard model, encompassing a large class of random graphs which keep some of the nice properties of Erd\H{o}s-R\'enyi graphs but which approximate real-life networks much more closely. Some standard examples include the following. 

\begin{itemize}
    \item The Chung-Lu model, where we fix a probability distribution $\nu$ over $\R_+$, sample iid random variables $(d_v)_{v \in [n]} \sim \nu$, and set
    \begin{equation}
        p_e = \min\left(1, \frac{d_u d_v}{\sum_{w \in [n]} d_w}\right) \text{ for any } e= \{ u, v \} \in E.
    \end{equation}
    This is a standard model used in the study of so-called scale-free graphs, where the degree $\deg v$ of a uniformly random vertex $v \in [n]$ follows a power-law distribution: for large $D$, $\Proba(\deg v \geq D) \simeq D^{-\gamma}$ for some fixed $\gamma > 0$. This is a significant difference from the Erd\H{o}s-R\'enyi model where the maximal degree is of logarithmic order.
    
    Here, the graph is weakly inhomogeneous iff $\E[\nu] < +\infty$. (For a proof, see the proof of Corollary \ref{chunglu} in Appendix \ref{corollaryproofs}.)
    \\

    \item Noisy random geometric graphs, defined as follows. Fix a metric probability space $(X, d, \mu) = (X_n, d_n, \mu_n)$, a radius $r = r_n > 0$ and a sampling parameter $p = p_n > 0$; we sample $n$ iid points $(x_1, \ldots, x_n)$ from the law $\mu_n$, and set 
    \begin{equation}
        p_e = p_n \1_{d(x_u, x_v) \leq r} \text{ for any } e= \{ u, v \} \in E.
    \end{equation}

    This is often used to model high-dimensional data (a classical choice is the sphere $X_n = \mathbb{S}^{k_n}$ for some adequate sequence $(k_n)$).

    Here, the graph is weakly inhomogeneous if $p_n = O(\frac{1}{n})$. 
    \\

    \item The stochastic block model, where the vertices are partitioned into unknown communities $V_1, \ldots, V_m$ for some $m \geq 1$, and we fix a symmetric matrix $(q_{ij})_{1 \leq i,j \leq m}$; then, 

    \begin{equation}
        p_{\{u,  v \}} = q_{ij} \text{ for any } u \in V_i, v \neq u \in V_j.
    \end{equation}
    As the previous description suggests, this is often used to model networks with a latent community structure. 

    Here, the graph is weakly inhomogeneous iff $\max_{i,j} q_{ij} = O(\frac{1}{n})$. 
    \\
\end{itemize}

Finally, let us define our model for correlated inhomogeneous random graphs.

\begin{definition}
    Let $G^{\uparrow}$ be an inhomogeneous graph and $s \in [0, 1]$ (which may depend on $n$). Then, we define a \textbf{correlated inhomogeneous graph system} $(G, H)$ by subsampling $G^{\uparrow}$, as follows. Sample variables $(\mc{B}_e^{K_1})_{e \in E}, (\mc{B}_e^{K_2})_{e \in E}$ following a $Ber(s)$ law independently from each other and anything else, and define $K_1, K_2$ by 
\begin{equation}
    \1_{e \in K_1} = \mc{B}_e^{K_1} \1_{e \in G^{\uparrow}}, \qquad \1_{e \in K_2} = \mc{B}_e^{K_2} \1_{e \in G^{\uparrow}}.
\end{equation}

Then, as in Definition \ref{graphalignment}, we sample $\sigma^* \in \Sym_n$ uniformly at random and set $G = (\sigma^*)^{-1}(K_1), H = K_2$ to obtain our graph system.
\end{definition}

This definition is standard; see \cite{raczsridhar2023, newah1}.

Before moving on, let us prove a fact about weakly inhomogeneous graphs which will serve us later on.

\begin{proposition} \label{weakchar}
    Let $K$ be a weakly inhomogeneous random graph, with associated edge probabilities $(p_e)_{e \in E}$. Then, for any sequence $\lambda_n \to +\infty$ with $\lambda_n \leq n$,
    \begin{equation}
        \#\{e \in K: p_e \geq \frac{\lambda_n}{n} \} = o_\Proba(n).
    \end{equation}

    In particular, there exists $\overline{K} \sim \mc{G}(n, \frac{\lambda_n}{n})$ such that $K \lesssim_{n \to +\infty} \overline{K}$.
\end{proposition}

\begin{proof}
    Assume that $K$ is weakly inhomogeneous. Then, for any $\lambda_n \to +\infty$, 

    \begin{equation}
    \begin{split}
        \E[ \#\{e \in K: p_e \geq \frac{\lambda_n}{n} \}] &= \E\left[\sum_{e \in E: p_e \geq \frac{\lambda_n}{n}} \Proba(e \in K | (p_e)) \right] \\
        & = \binom{n}{2}\E_{\mathbf{e} \sim \Unif(E)} [p_{\mathbf{e}} \1_{p_e \geq \frac{\lambda_n}{n}}] \\
        & = o(n),
        \end{split}
    \end{equation}
    by uniform integrability.
\end{proof}

As we shall soon see, this fact will allow us to generalise the properties of Erd\H{o}s-R\'enyi graph alignment to weakly inhomogeneous graphs.

\section{Main results} \label{mainresults}

In order to state our main results, we will need to define a couple of concepts. Let us begin with the balanced load function associated with a graph, first introduced in \cite{hajek1990}.

\begin{definition}
     Let $K$ be a graph over $V$. An \textbf{allocation} on $K$ is a map $\theta: \vec{E}(K) \to [0, 1]$ such that, for any $(u \to v) \in \vec{E}(K)$, $\theta(u \to v) + \theta(v \to u) = 1$. Any such allocation induces a so-called \textbf{load} map
    \begin{equation}
    \begin{array}{c}
\partial \theta: 
\end{array} \hspace{5pt}
\begin{array}{rcl}
V & \longrightarrow & \R_+ \\
v & \longmapsto & \sum_{\{u, v\} \in E(K)} \theta(u \to v)
\end{array}.
\end{equation}

An allocation $\theta$ is said to be \textbf{balanced} if for any $e = \{u, v\} \in E(K)$, 
\begin{equation}
    \partial \theta(u) < \partial \theta(v) \Longrightarrow \theta(u \to v) = 0.
\end{equation}

Then, balanced allocations over $K$ exist, and they all share the same associated load function $w_K$. We call this map the \textbf{balanced load function} associated to $K$.
\end{definition}

For the moment, we will simply note that, informally, the function $w_K(v)$ represents the "local density" of $K$ around a given vertex $v \in V$: a vertex $v$ which is more strongly connected to the rest of the graph will tend to have a higher associated balanced load. (For instance, on a $d$-regular graph the balanced load is constant and equal to $\frac{d}{2}$.) We will provide further context on this function in Section \ref{balancedload}.

Our second definition is the notion of self-overlap of a random graph.

\begin{definition}
    Let $K$ be a graph over $V$. We define the self-overlap of $K$:

\begin{equation} 
    SOV(K) = \operatorname*{lim}_{\eps \to 0} \operatorname*{plimsup}_{n \to +\infty}  \max_{\sigma \in \Sym_n}\max_{\underset{|\mc{A}| \geq \eps n}{\mc{A} \subseteq \Der \sigma}}  \operatorname*{min}_{v \in \mc{A}} w_{(K \land_\sigma K)}(v).
\end{equation}
\end{definition}

This is a scalar $\rho \in [0, +\infty]$ associated with any random graph $K$ which we introduce. In a certain sense, it encapsulates how self-symmetric the graph $K$ is. We will provide an in-depth explanation of this concept in Section \ref{sov}; for now, we will simply highlight the following points.
\begin{itemize}
    \item If $K$ has an automorphism with no fixed points,
    \begin{equation}
        SOV(K) = \min_{\mc{N} \subseteq V, |\mc{N}| = o_\Proba(n)} \max_{v \notin \mc{N}} w_K(v),
    \end{equation}
    and in particular $SOV(K) = \frac{d}{2}$ if $K$ is also $d$-regular. This follows directly by unpacking the definition of $SOV$; we omit the details.

    \item Informally, if $K$ is sparse ($O(1)$ average degree), and is "very asymmetric", then we should expect to have $SOV(K) = 1$. (For more precise statements, see Corollary \ref{weaksov} and Proposition \ref{starsnstripes}.)
\end{itemize}

Finally, we will fix the following notations.

\begin{definition}
    Let $G, H$ be correlated random graphs. From here on out, we set
    \begin{equation}
        I = G \land_{\sigma^*} H \qquad \text{and} \qquad U = G \lor_{\sigma^*} H.
    \end{equation}
    These are the true intersection and union graphs associated with the system. 
    Furthermore, if $t \geq 0$, we set
\begin{equation}
     V_{\geq t} = V_{\geq t}(I) = \{ v \in V: w_I(v) \geq t \}; \qquad V_{\leq t} = V_{\leq t}(I) = \{ v \in V: w_I(v) \leq t \}
\end{equation}
the level sets for the balanced load function associated with $I$.
\end{definition}

\subsection{Feasibility of graph alignment}

The following theorem is the main motivation for our definition of $SOV$.

\begin{theorem} \label{possible}
    Let $(G, H)$ be correlated random graphs, and set $\rho = SOV(U)$. Then, for any $\eps > 0$, there exists an estimator $\hat{\sigma}$ such that 

    \begin{equation}
        ov(\hat{\sigma}, \sigma^*) \geq \frac{|V_{\geq (\rho + \eps)}|}{n}(1 - o_\Proba(1)).
    \end{equation}
\end{theorem}

Note that this theorem does not assume anything about the graphs $(G, H)$. The only characteristics of the model which we must understand in order to apply this are:
\begin{itemize}
    \item the behaviour of the intersection $I = G \land_{\sigma^*} H$, which is usually already known if the graph model is standard;
    \item the quantity $SOV(U)$, which we will soon enough provide paths to compute.
\end{itemize}

This also fits with the intuitive idea that the more self-similarity the system exhibits, the more difficult the graph alignment problem becomes. Indeed, all else being equal, a larger value of $\rho=SOV(U)$ raises the balanced-load threshold for recoverability and therefore reduces the set $V_{\geq(\rho+\eps)}$ guaranteed to be alignable by the theorem.

Let us prove this theorem immediately.

\begin{proof}
    If $\sigma \in \mc{S}_n$, define $I(\sigma) = G \land_\sigma H$ (so that $I(\sigma^*) = I$). 
        Let $\eta > 0$, and let $\hat{\sigma} \in \mc{S}_n$ be such that $|\widehat{V_{dense}}| \defeq \# \{ v \in [n]: w_{I(\hat{\sigma})}(v) \geq \rho + \eps \}$ is maximal among all such $\sigma$. (In particular, $|\widehat{V_{dense}}| \geq |V_{\geq (\rho + \eps)}|$.) We will prove that $\hat{\sigma}$ fits, by showing that the set

\begin{equation}
    N = \{v \in  \widehat{V_{dense}}, \hat{\sigma} (\sigma^*)^{-1}(v) \neq v \}
\end{equation}
has size $|N| = o_\Proba(n)$.

        To this end, let us assume by contradiction that for some $\eta, \eta' > 0$, for large enough $n$, $|N| \geq \eta n $ with probability at least $\eta'$. Then, since $I(\hat{\sigma} )\subseteq U \land_{\hat{\sigma} (\sigma^*)^{-1}} U$, by monotonicity of $w$ (Proposition \ref{balancedload3}), 

        \begin{equation}
            \forall v \in N, w_{U \land_{\hat{\sigma} (\sigma^*)^{-1}} U}(v) \geq w_{I(\hat{\sigma} )}(v) \geq \rho + \eps.
        \end{equation}
        Since $N \subseteq \Der (\hat{\sigma} (\sigma^*)^{-1})$, this implies that
        
        \begin{equation}
            SOV(U) \geq \operatorname*{plimsup}_{n \to +\infty} \left(\1_{|N| \geq \eta n} \ds\min_{v \in N} w_{U \land_{\hat{\sigma} (\sigma^*)^{-1}} U}(v) \right) \geq \rho + \eps
        \end{equation}
        
        which we know to be false; contradiction.
\end{proof}

\begin{remark} \text{ }

\begin{itemize}
    \item  The estimator above is very similar in essence to the so-called $k$\textbf{-core estimator}, which has been extensively studied as a way to match graphs (see \cite{raczsridhar2023}). Indeed, the two definitions mirror each other, and Proposition \ref{balancedload2} provides explicit links between the $k$-core of a graph and the subsets $V_{\geq t}$ of sufficient load. In Section \ref{corollaries}, we will provide some corollaries of Theorem \ref{possible} formalising this remark.
    \item   The proof shows us that we can also recover a subset $\widehat{V_{dense}}$ such that $|\widehat{V_{dense}}| \geq |V_{\geq(\rho + \eps)}|$ and $\hat{\sigma}$ only makes $o_\Proba(n)$ mistakes over $\widehat{V_{dense}}$. However, counterintuitively, we have no guarantee a priori that $V_{\geq(\rho + \eps)} \subseteq \widehat{V_{dense}}$, even up to $o_\Proba(n)$ vertices. We discuss this further in Appendix \ref{discussion}.
\end{itemize}
\end{remark}

\subsection{Sharpness and infeasibility of graph alignment}

Let us continue by characterising the self-overlap in the case of Erd\H{o}s-R\'enyi graphs.
\begin{theorem} \label{ersov}
    Let $\lambda > 0$ and $\frac{1}{2} < \alpha \leq 1$ be fixed, and let $K \sim \mc{G}(n, \frac{\lambda}{n^{\alpha}})$ be an Erd\H{o}s-R\'enyi graph. Then,

    \begin{equation}
        SOV(K) = \frac{1}{2 \alpha -1}.
    \end{equation}
\end{theorem}

We prove this theorem in Section \ref{proofersov}.

This theorem serves a dual purpose. The first purpose is to show that, in the Erd\H{o}s-R\'enyi case, Theorem \ref{possible} is sharp: it recovers the same feasibility result as in \cite[Theorem 1.1]{du2025}, which is known to give the largest alignable subset (barring some technical details which we discuss in Remark \ref{1weightcase}). 

However, this identity can also be used outside of the Erd\H{o}s-R\'enyi model, owing to some convenient monotonicity properties of the self-overlap. Indeed, it implies the following corollary.

\begin{corollary} \label{weaksov}
    Let $K$ be a weakly inhomogeneous graph. Then, $SOV(K) \leq 1$. In particular, if $K$ has a giant component, then $SOV(K) = 1$.
\end{corollary}

The first statement follows from Theorem \ref{ersov} and Proposition \ref{weakchar}, by noting that $SOV(K)$ is a monotonous function of $K$ (see Lemma \ref{monot} from Section \ref{sov2}). The second then follows by applying Proposition \ref{starsnstripes}.

As a result, in the case of weakly inhomogeneous graphs, we have a good understanding of the threshold provided by Theorem \ref{possible}. Our final main result will be to show that, in a certain sense, this threshold is also sharp.

\begin{theorem} \label{impossible}
    Let $(G^\uparrow, G, H)$ be an inhomogeneous correlated graph system, and assume that $G^\uparrow$ is weakly inhomogeneous. Define
    
    \begin{equation}
    x_{sparse} = \lim_{\eps \to 0} \operatorname*{pliminf}_{n \to +\infty} \frac{|V_{\leq (1-\eps)}|}{n}.
\end{equation}
    Furthermore, assume that for some $r \geq 0$,
    \begin{equation} \label{noisycond}
        \min_{e \in E} p_e \geq \frac{1}{n (\log n)^r} \qquad \text{whp.}
    \end{equation}
    Then, if $\hat{\pi}$ is an estimator of $\pi^*$, there exist at least $x_{sparse}(n - o_\Proba(n))$ vertices $v \in [n]$ such that $\hat{\pi}(v) \neq \pi^*(v)$.
\end{theorem}
We prove this theorem in Section \ref{proofimpossible}.

Note that, here, $V_{\leq(1-\eps)}$ is easy to describe: it is the set of vertices whose connected component in $I$ is a tree of size at most $\frac{1}{\eps}$. (We prove this in Proposition \ref{balancedload2}.) In particular, just like for  Erd\H{o}s-R\'enyi graphs, graph alignment becomes completely intractable in this case if $I$ is composed mostly of trees of size $O(1)$.

Condition (\ref{noisycond}) can be viewed as a sort of mild "noise-robustness" condition; it is similar in spirit to the noise-robustness criterion in \cite{brennan2021}, ensuring that our inference problem is in a certain sense "generic". Informally, it enforces the fact that the observer cannot use the information of which edges are \textbf{not} present in the intersection graph $I(\pi)$ in order to deduce whether we could have $\pi^* \simeq \pi$. It is unclear to the author whether or not this theorem still holds if we remove condition (\ref{noisycond}).

\begin{remark} \label{1weightcase}
    Even in the weakly inhomogeneous case, theorems \ref{possible} and \ref{impossible} do not necessarily fully characterise the maximal alignable subset of $V$, even up to $o_\Proba(n)$ vertices. Indeed, these theorems say nothing about the vertices $v \in V$ such that $w_I(v) = 1$. This is not a negligible subset: for instance, if $I$ has a giant component, then all vertices $v$ contained in the giant component but not contained in the $2$-core will verify $w_I(v) = 1$.

    This ambiguity already appeared in the Erd\H{o}s-R\'enyi case, and is briefly discussed in \cite{du2025}. In particular, they comment that, if the subsampling parameter $s$ is set to $1$, we can align a positive proportion of these weight-$1$ vertices, but not all of them. It seems plausible that this may be the case for generic weakly inhomogeneous graphs; we leave this as a direction for future work.
\end{remark}

\section{Some consequences of our results} \label{corollaries}

In this section, we give some corollaries which follow from our main results. Any proofs are deferred to Appendix \ref{corollaryproofs}. Throughout, we fix a correlated graph system $(G, H)$, and denote by $\hat{\pi}$ the estimator given by Theorem \ref{possible}.

We will begin by specialising Theorem $1$ in order to explain, in concrete graph-theoretic terms, when $\hat{\pi}$ recovers the latent matching $\pi^*$ on a macroscopic subset of vertices. 

\begin{corollary} \label{giant}
    Assume that $SOV(U) = 1$. Assume that $I$ has a connected component $\mc{G}$, with vertex set $V_\mc{G}$, such that there exist $\eps, \eta > 0$ such that, with high probability,

    \begin{equation} \label{techgiant}
        \forall N \subseteq V_\mc{G} \text{ with } |N| \leq \eta n, \qquad |I[V_\mc{G} \setminus N]| \geq (1 + \eps )|V_\mc{G}|.
    \end{equation}
    Then, $\ds \operatorname*{pliminf}_{n \to +\infty} ov(\hat{\pi}, \pi^*) > 0$.
\end{corollary}

In practice, we will be using Corollary \ref{giant} to determine the threshold for feasibility of partial alignment: if the intersection graph $I$ has a giant component, then (\ref{techgiant}) is usually verified.

Theorem \ref{possible} also specialises into the following statement.

\begin{corollary} \label{cores}
    Let $k \geq 3$, and assume that  $SOV(U) < \frac{k}{2}$. Then,

            \begin{equation}
                ov(\hat{\pi}, \pi^*) \geq \frac{|\Ker_k(I)|}{n}(1 - o_\Proba(1)),
            \end{equation}
        
        where $\Ker_k(I)$ is the set of vertices inside the $k$-core of $I$.
\end{corollary}
\begin{remark}
    Corollary \ref{cores} also technically applies when $k=2$; however, using Proposition \ref{starsnstripes} one may see that it is a very weak statement in this case. Corollary \ref{giant} thus serves as a substitute.
\end{remark}

Using Theorem \ref{ersov}, we may also recover the following result.

\begin{corollary} \label{compare}
    Assume that $(G, H)$ is an inhomogeneous graph system, and assume that $\max_{e \in E} (s p_e) \leq n^{-\alpha + o(1)}$ for some $\frac{1}{2} < \alpha \leq 1$. Denote by $\hat{\pi}$ the estimator given by Theorem \ref{possible}. Then, for any $k > \frac{2}{2\alpha - 1}$, $ov(\hat{\pi}, \pi^*) \geq \frac{|\Ker_k(I)|}{n}(1 - o_\Proba(1))$.

    In particular, if $\ds\min_{u \in V} \left( \sum_{v \in V}  (p_{\{ u, v\}}s^2)\right) \goesto{n \to +\infty} +\infty$, then $ov(\hat{\pi}, \pi^*) = 1 - o_\Proba(1)$: almost exact alignment is feasible.
\end{corollary}

We highlight this result because it is directly comparable to \cite[Lemma III.4]{raczsridhar2023}; they use a $k$-core estimator to prove a very similar result. The differences between the two results are the following: 

\begin{itemize}
    \item the authors of \cite{raczsridhar2023} obtain \textbf{exact} guarantees that all of the vertices which they align are correctly matched (with no $o_\Proba(n)$ error tolerance);
    \item their Lemma III.4 requires that $k > \frac{12}{2 \alpha -1}$, whereas we simply need $k > \frac{2}{2 \alpha -1}$.
\end{itemize}

Let us emphasise, however, that the main purpose of this paper is not to slightly improve upon the specific $k$-cores which we are capable of aligning; rather, our framework is primarily intended to be more flexible and not to rely on extreme value $\max_{e \in E} (sp_e)$ as a proxy for the behaviour of the graph. To illustrate this, let us specialise our results to the Chung-Lu model.

\begin{corollary} \label{chunglu}
    Assume that $(G, H)$ is an inhomogeneous graph system, with a constant subsampling parameter $s \in (0, 1]$, where the mother graph $G^\uparrow$ follows the Chung-Lu model, with degree-law $\nu \in L^1$. Assume that $\nu \neq \delta_0$, and set 

    \begin{equation}
        d^* = \frac{\E[D^2]}{\E[D]} \in (0, +\infty]
    \end{equation}
    where $D \sim \nu$ is the degree-law of the model. Then:
    \begin{itemize}
        \item If $s^2 > \frac{1}{d^*}$, then $ \operatorname{pliminf}_{n \to +\infty}ov(\hat{\pi}, \pi^*) > 0$: partial alignment is feasible.
        \item If $s^2 < \frac{1}{d^*}$, then as long as $\nu$ is supported on $[\eps, +\infty)$ for some fixed $\eps > 0$, for any estimator $\hat{\pi}
        '$, $ \ds\operatorname{plimsup}_{n \to +\infty}ov(\hat{\pi}', \pi^*) = 0$: partial alignment is intractable.
    \end{itemize}
\end{corollary}

In this situation, $G^\uparrow$ is weakly inhomogeneous, so that $SOV(U) \leq 1$; thus, this threshold lines up which the threshold for the existence of a giant component (and a $2$-core) within the graph $I$. In particular, we will note that if $\E[D^2] = +\infty$, then the threshold is at $0$: there is always a giant component in the intersection graph, and partial alignment is always feasible. 

We may also state a similar theorem for the stochastic block model.

\begin{corollary} \label{sbm}
    Assume that $(G, H)$ is an inhomogeneous graph system, with a constant subsampling parameter $s \in (0, 1]$, where the mother graph $G^\uparrow$ follows the stochastic block model. If $V_i, V_j$ are communities, denote $\alpha_i = \frac{|V_i|}{n}$ the size of community $i$ and $\frac{q_{ij}}{n}$ the edge probability from $V_{i}$ to $V_j$. We then define $\lambda = \lambda_n$ to be the Perron-Frobenius eigenvalue of the matrix 
    \begin{equation}
        Q D = (q_{ij})_{1 \leq i,j \leq m} \cdot \diag(\alpha_1, \ldots, \alpha_m).
    \end{equation}
    Assume that for any $i, j$, $q_{ij} > 0$ is a constant, independent from $n$. Then: 
       \begin{itemize}
        \item if $\ds\operatorname*{pliminf}_{n  \to +\infty} (s^2 \lambda_n) > 1$, then $\ds\operatorname*{pliminf}_{n \to +\infty}ov(\hat{\pi}, \pi^*) > 0$: partial alignment is feasible.
        \item if $\ds\operatorname*{plimsup}_{n  \to +\infty} (s^2 \lambda_n) < 1$, then for any estimator $\hat{\pi}'$, $ \ds\operatorname*{plimsup}_{n \to +\infty}ov(\hat{\pi}', \pi^*) = 0$: partial alignment is intractable.
    \end{itemize}
\end{corollary}

A similar result also holds for noisy random geometric graphs in the weakly inhomogeneous case (though the threshold will of course depend upon the underlying sequence of metric spaces).

\section{The balanced load function} \label{balancedload}
In this section, $K$ is a graph over $V$.

As we have seen, an important object in our study of the graph alignment problem is the balanced load function $w$. In this section, we recall some of its basic properties. All proofs are deferred to the appendix.

Balanced allocations were initially studied in \cite{hajek1990} in the context of resource allocation problems; however, they have since been applied to a broader class of problems, including identification of the densest subgraph \cite{anantharam2016} and, more recently, graph alignment \cite{dingdu2023detection, du2025}. Indeed, this function provides a notion of "local density" of $K$ around a given vertex $v \in V$ which turns out to be quite helpful. The following proposition is a manifestation of this fact.

If $t \geq 0$, we define
\begin{equation}
        V_{\geq t} = V_{\geq t}(K) = \{ v \in V: w_K(v) \geq t \}; \qquad V_{< t} = \{ v \in V: w_K(v) < t \}.
    \end{equation}

\begin{proposition} \label{balancedload1}
    Let $t \geq 0$.
    Then, if $X \subseteq V_{\geq t}$,

    \begin{equation}
       \sum_{u \in X} w_K(u) \leq |K[{X \leftrightarrow V_{\geq t}}]| .
    \end{equation}
    Similarly, if $Y \subseteq V_{< t}$,
    \begin{equation}
        \sum_{u \in Y} w_K(u) \geq |K[{Y \leftrightarrow (V_{\geq t} \cup Y)}]|.
    \end{equation}

    In particular, $\dsum_{u \in V_{\geq t}} w_K(u) = |K[V_{\geq t}]|$.
\end{proposition}

In certain cases, the sets $V_{\geq t}$ can also give us extra information regarding certain structural properties of $K$.

\begin{proposition} \label{balancedload2}
    \begin{enumerate}
        \item Let $\eps > 0$. Then, the subgraph $K[{V_{\leq(1-\eps)}}]$ is a union of connected components of $K$, each of which is a tree supported on a vertex set of size at most $\frac{1}{\eps}$.
        \item Let $k \geq 2$, and denote by $\Ker_k (K)$ the set of vertices inside the $k$-core of $K$. Then, $V_{\geq k} \subseteq \Ker_k(K) \subseteq V_{\geq \frac{k}{2}}$.
       
        \item Let $\rho_m = \max_{X \subseteq V} \dfrac{|K[X]|}{|X|}$. Then, $\rho_m = \max_{v \in V} w_K(v)$, and the set $X = V_{\geq \rho_m}$ maximises $\dfrac{|K[X]|}{|X|}$.

    \end{enumerate}
\end{proposition}

Finally, we will note some monotonicity and stability properties of $w$, which are not a priori obvious.

\begin{proposition}\label{balancedload3}
    Let $K'$ be a graph over $V$. Then:
    \begin{itemize}
        \item if $K \subseteq K'$, then $w_K \leq w_{K'}$;
        \item if $|K \Delta K'| \leq \eps |V|$ for some $\eps > 0$, then there exists $\mc{A} \subseteq V$ such that $|\mc{A}| \leq 2 \sqrt{\eps} |V|$ and, for $v \notin \mc{A}$, $|w_K(v) - w_{K'}(v)| \leq 2\sqrt{\eps}$.
    \end{itemize}
\end{proposition}

\begin{remark}
    Though $w_K(v)$ can be conceptualised as the "local density" of $K$ at $v$, it is important to keep in mind that it is not actually a local function of $v$ in any reasonable sense. For instance, there exist sequences $(K_n, v_n)_{n \in \N}$ and $(K_n', v_n')_{n \in \N}$ of rooted graphs such that both rooted graphs have the same local weak limit, but $\lim_n|w_{K_n}(v_n) - w_{K'_n}(v'_n)|$ is arbitrarily large.

    This has important computational repercussions. Indeed, if $w_K(v)$ only depended on small neighbourhoods of $v$, we believe that our estimator from Theorem \ref{possible} could likely be modified to be computable in quasi-polynomial time. What we instead observe (for instance in the Erd\H{o}s-R\'enyi case) is an informational/computational gap: there are situations where said estimators are correct but the graph alignment problem remains intractable in polynomial time.
\end{remark}

\section{Self-overlap} \label{sov}
In this section, we assume that $K = (K_n)_{n \in \N}$ is a sequence of random graphs over $V$.
\subsection{Motivating the self-overlap}

Recall that we have defined
\begin{equation}
    SOV(K) = \operatorname*{lim}_{\eps \to 0} \operatorname*{plimsup}_{n \to +\infty}  \max_{\sigma \in \Sym_n}\max_{\underset{|\mc{A}| \geq \eps n}{\mc{A} \subseteq \Der \sigma}}  \operatorname*{min}_{v \in \mc{A}} w_{(K \land_\sigma K)}(v).
\end{equation}

Let us try to motivate this definition a little bit.

In a certain sense, this parameter quantifies how macroscopically self-similar the graph $K$ is. It is conceptually similar to the simpler statistic
\begin{equation} \label{tilde1}
    SOV^{(1)}(K) =  \operatorname*{plimsup}_{n \to +\infty} \max_{\sigma \in \mc{D}_n} \frac{|K \land_\sigma K|}{n},
\end{equation}
where $\mc{D}_n$ is the set of permutations $\sigma \in \Sym_n$ with no fixed points.

However, this simpler statistic is simply not robust enough to capture the information we care about: for instance, we can reduce $ SOV^{(1)}(K)$ by adding many isolated vertices to $K$, which does not make the graph any less self-similar. One fix might be to look at the more uniform statistic

\begin{equation} \label{tilde2}
     SOV^{(2)}(K) = \operatorname*{lim}_{\eps \to 0} \operatorname*{plimsup}_{n \to +\infty}  \max_{\sigma \in \Sym_n}\max_{\underset{|\mc{A}| \geq \eps n}{\mc{A} \subseteq \Der \sigma}}  \frac{|(K \land_\sigma K)[\mc{A}]|}{|\mc{A}|}
\end{equation}
and this is much closer to our definition of $SOV$. However, (\ref{tilde2}) fails to account for edges between $\mc{A}$ and $\mc{A}^c$: this is important, since in sparse graphs the overlap will often be carried by these outgoing edges. We can repair this in the following fashion: if $\theta_\sigma$ is a balanced allocation on $K \land_\sigma K$, we may define

\begin{equation} \label{tilde3}
     SOV^{(3)}(K) = \operatorname*{lim}_{\eps \to 0} \operatorname*{plimsup}_{n \to +\infty}  \max_{\sigma \in \Sym_n}\max_{\underset{|\mc{A}| \geq \eps n}{\mc{A} \subseteq \Der \sigma}} \frac{1}{|\mc{A}|} \sum_{u \in \mc{A}} \sum_{v \in V} \theta_\sigma(v \to u).
\end{equation}

which adds the edges linking $\mc{A}$ to $\mc{A}^c$ to the calculation, with an "adequate" weighting taking into account the general density of the graph. In fact, $SOV^{(3)}$ is almost the same as $SOV$: we have simply replaced the term $\ds\operatorname*{min}_{v \in \mc{A}} w_{(K \land_\sigma K)}(v)$ in the definition of $SOV$ with $\ds\frac{1}{|\mc{A}|} \sum_{v \in \mc{A}} w_{(K \land_\sigma K)}(v)$. (In particular, $SOV^{(3)} \geq SOV$.)

Of course, this discussion doesn't really explain why $SOV$ is an interesting definition for the purpose of graph alignment. For this, we point towards Theorem \ref{possible} and its proof, in which the quantity $SOV(K)$ appears fairly naturally.

\subsection{Properties of $SOV$} \label{sov2}

We begin with the following basic properties.

\begin{lemma} \label{monot}
    Let $K' = (K'_n)$ be a sequence of random graphs with vertex set $[n]$.
    \begin{itemize}
         \item If $K \simeq_{n \to +\infty} K'$ then $SOV(K) = SOV(K')$.
        \item If $K \subseteq K'$ whp then $SOV(K) \leq SOV(K')$.
    \end{itemize}
\end{lemma}

\begin{proof}
    These follow immediately from Proposition \ref{balancedload3}.
\end{proof}

We will make extensive use of these monotonicity/stability properties later on.

At this point, the reader may be wondering what the quantity $SOV(K)$ actually tends to looks like. The following result is a first step towards understanding this.

\begin{proposition} \label{starsnstripes}
    Assume that, with high probability, $K$ has a giant component (i.e. a connected component of size $\Theta_\Proba(n)$). Then, $SOV(K) \geq 1$.
\end{proposition}

If $K$ is a very large tree, then the proposition is in essence stating that we can choose $\sigma \in \mc{D}_n$ such that $K \land_\sigma K$ contains a very dense forest as a subgraph. This should not be surprising; one may for example imagine a greedy algorithm, which iteratively constructs $\sigma(v)$ for $v \in [n]$ by sending it to a neighbour of $\sigma(u)$ where $u$ is a neighbour of $v$, and such an algorithm would likely not do too badly. We will however note that in the proof (which we defer to the appendix), we use a slightly different approach to construct such a $\sigma$.

\section{Proof of Theorem \ref{ersov}} \label{proofersov}
In this section, we prove Theorem \ref{ersov}. We will proceed by bounding $SOV(K)$ on both sides.

\subsection{$SOV(K) \leq \frac{1}{2 \alpha - 1}$}
We will assume that $\alpha < 1$; the case $\alpha=1$ follows by taking limits. 

Our main ingredient will be the following lemma.

\begin{lemma} \label{edgebounds}
    Let $\eps, \eta > 0$; then, the following holds with high probability. For any $\sigma \in \Sym_n$, for any $X \subseteq \Der \sigma$ with size at least $\eps n$, 

    \begin{equation}
        \left|(K \land_\sigma K)\left[X \leftrightarrow (X \sqcup \Fix \sigma)\right]\right| \leq \left( \frac{1}{2 \alpha -1} + \eta \right) |X|.
    \end{equation}
    ie. the number of edges with one endpoint in $X$ and the other endpoint within $X \cup \Fix \sigma$ is bounded by $\left( \frac{1}{2 \alpha -1} + \eta \right) |X|$.
\end{lemma}

Let us first explain why this lemma implies $SOV(K) \leq \frac{1}{2 \alpha - 1}$. Indeed, for $\eta > 0$, and for any $\sigma \in \Sym_n$, let us now set 

\begin{equation}
    V_{\geq (\frac{1}{2\alpha - 1} + 2\eta)} = V_{\geq (\frac{1}{2\alpha - 1} + 2\eta)}(K \land_\sigma K); \qquad X = X(\sigma) = \Der \sigma \cap \left( V_{\geq (\frac{1}{2\alpha - 1} + 2\eta)} \right)
\end{equation}

By Proposition \ref{balancedload1},

\begin{equation}
\begin{split}
    \left|(K \land_\sigma K)\left[X \leftrightarrow (X \sqcup \Fix \sigma)\right]\right| & \geq \left|(K \land_\sigma K)\left[X \leftrightarrow \left(V_{\geq (\frac{1}{2\alpha - 1} + 2\eta)}\right)\right]\right| \\
    & \geq \sum_{v \in X} w_{K \land_\sigma K}(v)  \\
    & \geq \left( \frac{1}{2 \alpha -1} + 2\eta \right) |X|.
    \end{split}
\end{equation}

   By Lemma \ref{edgebounds}, this implies that, for any $\eps > 0$, the event $\{ \forall \sigma \in \Sym_n, |X(\sigma)| \leq \eps n \} $ holds with high probability. Thus, $SOV(K) \leq \frac{1}{2\alpha -1} + 2 \eta$ and $SOV(K) \leq \frac{1}{2\alpha -1}$.

We are now tasked with proving Lemma \ref{edgebounds}.

\begin{proof}[Proof of Lemma \ref{edgebounds}]

We begin with the following definition.

\begin{definition}
    Let $\sigma \in \Sym_n$, and $X \subseteq \Der \sigma$. Define
    \begin{equation}
        E^{(X, \sigma)} \defeq E[X \leftrightarrow (X \cup \Fix \sigma)] = \{e \in E \text{ with one endpoint in } X \text{ and the other in } X \cup \Fix \sigma \}.
    \end{equation}
    Then, if $e = \{u, v\} \in E^{(X, \sigma)}$, exactly one of the following holds.

    \begin{enumerate}
        \item For some $k \in \N$, $\sigma^k(e) \notin E^{(X, \sigma)}$. If this is the case, there exist maximal $i,j \geq 0$ such that $\{\sigma^{-i}(e), \sigma^{-(i-1)}(e) \ldots, e, \ldots, \sigma^j(e) \} \subseteq E^{(X, \sigma)}$; such a maximal subset is called a chain. \\
        \item The $\sigma$-orbit $\mc{O}$ of $e$ is contained within $E^{(X, \sigma)}$, i.e. the $\sigma$-orbits $\mc{O}_u$ (resp. $\mc{O}_v$) of $u$ (resp. $v$) are contained within $X \cup \Fix \sigma$. In this case:
        
        \begin{enumerate}
            \item either $|\mc{O}| = \mathrm{lcm}(|\mc{O}_u|, |\mc{O}_v|)$;
            \item or $\mc{O}$ is a so-called "special" orbit. 
        \end{enumerate}
        There are at most $n$ edges inside special orbits.
    \end{enumerate}
\end{definition}

The fact that this holds follows from some elementary graph theory; we refer to \cite[Proposition 5.1]{wuxuyu2023} for an explanation.

    Set $L = \lfloor \frac{1}{1 - \alpha} \rfloor$. If $\sigma \in \Sym_n$ and $X \subseteq \Der \sigma$, our job is to control the number of edges of $K \land_\sigma K$ which are contained inside $E^{(X, \sigma)}$. To this end, we will partition $E^{(X, \sigma)}$ into:
\begin{itemize}
    \item the subset $E_s$ containing the edges from special cycles;
    \item the subsets $E_k$ containing the edges from non-special cycles of length $k$, for $2 \leq k \leq L$;
    \item the subset $E_{>L}$ containing the edges from non-special cycles of length strictly greater than $L$;
    \item the subset $E_c$ containing the edges from chains.
\end{itemize}

(Note that, by construction, there are no non-special cycles of length $1$ inside $E^{(X, \sigma)}$.) 

We will proceed by controlling $e_s(\sigma, X) = |(K \land_\sigma K) \cap E_s|$, $e_k(\sigma, X) = |(K \land_\sigma K) \cap E_k|$, etc. In order to do this, we rely upon the following lemma, which is adapted from \cite{dingdu2023}; its proof is deferred to the appendix.

\begin{lemma}\label{tailboundsfixedsigma}
     Let $\rho \geq 0$; then, we have the following inequalities. For any $\sigma \in \mc{S}_n$, $X \in \Der \sigma$, 

    \begin{itemize}
        \item if $2 \leq k \leq L$, $\Proba(e_k(\sigma, X) \geq \rho n) \leq \exp\left(-(\alpha - \frac{1}{k}) \rho n \log n + o(n \log n)\right)$;
        \item $\Proba(e_{>L}(\sigma, X) \geq \rho n) \leq  \exp\left(-(2 \alpha - 1) \rho n \log n + o(n \log n)\right)$;
        \item $\Proba(e_s(\sigma, X) \geq \rho n) \leq \exp\left(-\alpha\rho n \log n + o(n \log n)\right)$;
        \item $\Proba(e_{c}(\sigma, X) \geq \rho n) \leq  \exp\left(-(2 \alpha - 1) \rho n \log n + o(n \log n)\right)$.
    \end{itemize}
\end{lemma}

Finally, we need to recombine the inequalities of Lemma \ref{tailboundsfixedsigma} into a global bound on the number of edges in $(K \land_\sigma K)[X \leftrightarrow (X \cup \Fix \sigma))]$. This is the object of the following sublemma, whose proof is again deferred to the appendix.

\begin{definition}
    Let $\sigma \in \Sym_n$ and $X \subseteq [n]$. The action of $\sigma$ on $X$ decomposes into orbits which are either $k$-cycles or chains. We define $\mc{C}_k(\sigma, X)$ to be the union of the $k$-cycles and $\mc{C}_c(\sigma, X)$ to be the union of the chains.
\end{definition}

\begin{lemma} \label{whatastatement}
    Let $\eta, \eps > 0$. Let $X \subseteq [n]$ with $|X| \geq \eps n$, and let $X_2, \ldots, X_L, X_{>L}, X_c$ form a partition of $X$. Denote $\kappa = \frac{|X|}{n}$. Let $\rho_2, \ldots, \rho_L, \rho_{>L}, \rho_{c}, \rho_s \geq 0$ be such that $\sum_{i} \rho_i \geq (\frac{1}{2 \alpha -1} + \eta) \kappa$. Then, there exists $\delta = \delta(\eta, \eps) > 0$ such that, for large enough $n$,

    \begin{equation}
        \Proba\left(\exists \sigma \in \Sym_n: X \subseteq \Der \sigma,\begin{cases}
            \mc{C}_2(\sigma, X) = X_2 \\
            \mc{C}_3(\sigma, X) = X_3 \\
            \hspace{29pt}\vdots \\
            \mc{C}_c(\sigma, X) = X_c
        \end{cases}, \begin{cases}
            e_2(\sigma, X) \geq \rho_2 n \\
            e_3(\sigma, X) \geq \rho_3 n \\
            \hspace{29pt}\vdots \\
            e_c(\sigma, X) \geq \rho_c n \\
            e_s(\sigma, X) \geq \rho_s n
        \end{cases} \right) \leq e^{-\delta n \log n}.
    \end{equation}
    
\end{lemma}

By union-bounding over all $(X_i, \rho_i = \frac{k_i}{n})$ (there are only $\exp O(n)$ terms in said union bound), this implies Lemma \ref{edgebounds}. 

\end{proof}

\subsection{$SOV(K) \geq \frac{1}{2 \alpha - 1}$}
Again, we will assume that $\alpha < 1$. For $\alpha = 1$, the result follows either from Proposition \ref{starsnstripes} (if $\lambda > 1$) or by adapting the proof of Proposition \ref{starsnstripes} (if $\lambda \leq 1$).

 Recall that, for $t \geq 0$, given a correlated graph system $(G, H)$, we have defined

    \begin{equation}
        V_{\geq t}= V_{\geq t}(I) = \{ v \in [n], w_I(v) \geq t\} \text{ resp. } V_{\leq t} = V_{\leq t}(I) = \{ v \in [n], w_I(v) \leq t\} .
    \end{equation}

In this section, we will make use of the following theorem, proven in \cite{du2025}.

\begin{theorem}[Du, 2025] \label{duthm}
    Fix $\lambda_1, \lambda_2 > 0$, $\beta \in (0, 1)$. Let $G^\uparrow \sim \mc{G}(n, \frac{\lambda_1}{n^\beta})$ and $s \sim \lambda_2n^{-\left(\frac{1-\beta}{2}\right)}$, and consider a correlated inhomogeneous graph system $(G, H)$ obtained by subsampling $G^\uparrow$ with parameter $s$. In particular,

    \begin{equation}
        I \sim \mc{G}(n, \frac{\lambda_1 \lambda_2^2}{n}); \qquad U \sim \mc{G}(n, \frac{\lambda_1 \lambda_2 (2-o(1))}{n^{\frac{1+\beta}{2}}})
    \end{equation}

    Then, for any $\eps > 0$: 

    \begin{itemize}
        \item there exists an estimator recovering $\sigma^*$ on at least $|V_{\geq (\beta^{-1} + \eps)}|(1 - o_\Proba(1))$ vertices;
        \item any estimator must be wrong on at least $|V_{\leq (\beta^{-1} - \eps)}|(1 - o_\Proba(1))$ vertices.
    \end{itemize}
\end{theorem}

We will compare this to Theorem \ref{possible} to show that, for well-chosen parameters $\lambda_1, \lambda_2$, and for any $\beta \in (0, 1)$,

\begin{equation}
    SOV(U) \geq \beta^{-1} = \frac{1}{2\frac{1+\beta}{2} - 1}.
\end{equation}

To this end, we make use of the following fact, which is proven in \cite{anantharam2016}.

\begin{proposition}
    Let $\kappa > 1$. Then, there exists $c(\kappa) > 1$ such that, if $J \sim \mc{G}(n, \frac{c(\kappa)}{n})$,

    \begin{equation}
        \rho_m(J) \defeq \max_{v \in [n]} w_J(v) \overset{(\Proba)}{\underset{n \to +\infty}{\longrightarrow}} \kappa.
    \end{equation}
    Furthermore, for any $\eps > 0$,
        $|V_{\geq (\rho_m - \eps)}(J)| = \Theta_\Proba(n).$
\end{proposition}
Now, assume by contradiction that $SOV(K) < \dfrac{1}{2 \alpha -1}$. Fix $\kappa \in \left(SOV(K),  \dfrac{1}{2 \alpha -1} \right)$, and let

\begin{equation}
    \lambda_1 = \dfrac{\lambda^2}{4 c(\kappa)}, \qquad \lambda_2 = \dfrac{2 c(\kappa)}{\lambda}, \qquad \beta = 2 \alpha -1.
\end{equation}

so that, by monotonicity, $SOV(U) \leq SOV(K)$. Then, under the notations of Theorem \ref{duthm}, if $\eps > 0$ is small enough then with high probability,

\begin{equation}
    SOV(U) < \rho_m(I) - 2\eps < \kappa < \rho_m(I) + 2\eps <  \dfrac{1}{\beta}.
\end{equation}

 By Theorems \ref{possible} and \ref{duthm}, there exists an estimator $\hat{\sigma}$ which recovers $\sigma^*$ on a subset of size at least $|V_{\geq(\kappa - \eps)}| (1 - o_\Proba(1)) = \Theta_\Proba(n)$ but is wrong on any subset of size larger than $|V_{\leq(\kappa + \eps)}| (1 - o_\Proba(1)) = n(1 - o_\Proba(1))$. This is a contradiction, concluding the proof.

\section{Proof of Theorem \ref{impossible}} \label{proofimpossible}

In this section, we prove Theorem \ref{impossible}.

 We may assume, without loss of generality, that 
\begin{equation} \label{wlog}
    |G|, |H| \leq n \sqrt{\log n}; \qquad\forall e \in E,  \frac{1}{n (\log n)^r} \leq p_e \leq 1 - n^{-3}
\end{equation}

since this does not affect the high-probability behaviour of the model.

For any $\pi \in \Sym_n, \eps > 0$, we set 

    \begin{equation}
        I(\pi) = G \land_\pi H; \qquad V_{\leq(1-\eps)}(\pi) \defeq V_{\leq(1-\eps)}(I(\pi)) =  \{ v \in [n], w_{I(\pi)}(v) \leq 1 - \eps\}.
    \end{equation}

We will use the following lemma in order to establish the theorem.

\begin{lemma} \label{bayesian}
    Let $x \geq 0$. Assume that, for any $\eta > 0$ and $\pi_0 \in \Sym_n$, there exist $\gamma > 0$, a random subset $\Neg \subseteq \Sym_n$, and a function $F: B(\pi_0, x - \eta) \setminus \Neg \to \mc{P}(\Sym_n)$ verifying the following conditions.

    \begin{enumerate}
        \item if $\pi \in \Dom F$, $|F(\pi)| \geq e^{\gamma n \log n}$;
        \item if $\pi \in \Dom F$ and $\pi' \in F(\pi)$, $\Proba(\pi'| G, H, (p_e)) \geq e^{-o_{\Proba}(n \log n)} \Proba(\pi| G, H, (p_e))$;
        \item if $\pi' \in \Sym_n$, $|F^{-1}(\pi')| \leq e^{o_\Proba(n \log n)}$ (uniformly over $\pi'$);
        \item $\Proba(\pi^* \in \Neg) = o(1)$.
    \end{enumerate}
    Then, if $\hat{\pi}$ is an estimator of $\pi^*$, there exist at least $x(n - o_\Proba(n))$ vertices $v \in [n]$ such that $\hat{\pi}(v) \neq \pi^*(v)$.
\end{lemma}

This lemma is a straightforward adaptation of Lemma 1 from \cite{vassauxmassoulie2026}; for a more in-depth discussion of the ideas behind it, we refer the interested reader to said paper. We also include a short proof of the lemma in Appendix \ref{impossibleproofs}.

In order to verify the prerequisites of Lemma \ref{bayesian}, we will make use of two additional sublemmas, whose proofs are deferred to the appendix.

\begin{lemma} \label{ppost}
    Assume that (\ref{wlog}) holds. Then, there exists an (random) subset $\Neg_1 \subseteq \Sym_n$ with $\Proba(\pi^* \in \Neg_1) = o(1)$ such that, if $\pi \in \Sym_n \setminus \Neg_1$ and $\sigma$ is an automorphism of $I(\pi)$, 

    \begin{equation}
        \Proba(\sigma\pi | G, H, (p_e)) \geq  \Proba(\pi | G, H, (p_e))e^{-o_\Proba(n \log n)} 
    \end{equation}
    uniformly in $\sigma, \pi$.
\end{lemma}

\begin{lemma} \label{negdef}
    Let $\delta, \eps > 0$. There exist constants $c = c_\eps \goesto{\eps \to 0} 0$, $C = C_\eps > 0$, $\gamma = \gamma_{\eps, \delta} > 0$ and a random subset $\Neg_2(\delta, \eps) \subseteq \Sym_n$, such that $\Proba(\pi^* \in \Neg_2) = o(1)$, and the following property holds.
    
    For any $\pi \notin \Neg_2$, $|V_{\leq(1 - \eps)}(\pi)| \geq x_{sparse}n (1 - c)$; furthermore, for any $A \subseteq V_{\leq(1 - \eps)}(\pi)$ with $|A| \geq \delta n$ and such that no edges of $I(\pi)$ link $A$ to $A^c$, the graph $I(\pi)[V_{\leq(1 - \eps)}(\pi)]$ has at least $\exp(\gamma n \log n)$ automorphisms $(\sigma_i)$, verifying $A^c \subseteq \Fix \sigma_i$.
\end{lemma}

Now, let $\eta > 0$ and $\pi_0 \in \Sym_n$. We pick $\eps > 0$ such that $c_\eps \leq \frac{\eta}{2}$ (as defined in Lemma \ref{negdef}), and we set $\delta = \frac{\eta }{2}$. We then define $\Neg = \Neg_1 \cup \Neg_2(\delta, \eps)$; we are going to build a function $F: B(\pi_0, x_{sparse} - \eta) \setminus \Neg \to \mc{P}(\Sym_n)$ to satisfy the requirements of Lemma \ref{bayesian}.

Let $\pi \in B(\pi_0, x_{sparse} - \eta) \setminus \Neg$. We consider the set $X(\pi) = \{ v \in V_{\leq(1 - \eps)}(\pi), \pi^{-1}(v) = \pi_0^{-1}(v) \}$; by Lemma \ref{negdef}, $|X(\pi)| \geq (\eta - c_\eps ) n$. We then consider the set $\overline{X}(\pi)$, which is the minimal subset of $V_{\leq(1 - \eps)}(\pi)$ containing $X$ such that no edges of $I(\pi)$ link $\overline{X}$ and $\overline{X}^c$. Since $\pi \notin \Neg$, by Lemma \ref{negdef}, for some $\gamma > 0$ we may construct $\exp(\gamma n \log n)$ automorphisms $(\sigma_1, \ldots, \sigma_N)$ of the graph $I(\pi)$ which leave $\overline{X}^c$ fixed. The function $F$ is defined as mapping $\pi$ to the set $\{ \sigma_i \pi, 1 \leq i \leq N \}$.

    Does $F$ satisfy our conditions? Conditions 1 and 4 are verified by construction; condition 2 is verified by Lemma \ref{ppost}. Condition 3, however, is less obvious; let us prove that it holds.

Let $\pi' \in \Sym_n$. We may partition $F^{-1}(\pi')$ as

\begin{equation}
    \begin{split}
        F^{-1}(\pi') & \subseteq \bigcup_{\underset{J \text{ subgraph of } \left(U(\pi')[Z]\right)}{Y \subseteq Z \subseteq [n]}} \{ \pi \in B(\pi_0, x_{sparse} - \eta) \setminus \Neg: X(\pi) = Y, \overline{X}(\pi) = Z, I(\pi)[Z] = J \} \\
            & \defeq\bigcup_{\underset{J \text{ subgraph of } \left(U(\pi')[Z] \right)}{Y \subseteq Z \subseteq [n]}} W(Y, Z, J).
    \end{split}
\end{equation}

We will first note that the indexation set has size $\exp O(n \sqrt{\log n})$. Indeed, there are at most $4^n$ choices for $(Y, Z)$; and, given that $|G|, |H| \leq 2 n \sqrt{\log n}$, the graph $U(\pi')[Z]$ has at most $O_\Proba(n \sqrt{\log n})$ edges (uniformly in $\pi'$), meaning that there are at most $\exp O_\Proba(n \sqrt{\log n})$ choices for $J$. We thus just need to prove that $|W(Y, Z, J)| \leq \exp O_\Proba(n)$ for all choices of $Z, J$.

Now, let $\pi_1, \pi_2 \in W(Y, Z, J)$; and let $\sigma = \pi_2 \pi_1^{-1}$. By construction, $\sigma = \Id$ over $Z^c$. Furthermore, $\sigma|_Z$ is an automorphism of $J$, which sends each connected component to itself (since $\sigma|_{Y} = \Id$). Since each connected component has vertex count at most $\frac{1}{\eps}$, there are only $\exp O_\Proba(n)$ possible choices for such a $\sigma$, and thus $|W(Z, J)| \leq \exp O_\Proba(n)$. 

Thus, we have proven that $F$ verifies point $3$, concluding the proof of Theorem \ref{impossible}.

\begin{appendix}
\section{Proofs of properties for the balanced load function}

\begin{proof}[Proof of Proposition \ref{balancedload1}]
    Let $\theta$ be a balanced allocation on $K$. Then, by definition, if $u \in V_{\geq t}$ and $v \in (V \setminus V_{\geq t})$, $\theta(v \to u) = 0$. Thus, if $X \subseteq V_{\geq t}$,

    \begin{equation}
        \begin{split}
        \sum_{u \in X} w_K(u) & = \sum_{\underset{v \in V \setminus \{ u\}}{u \in X}} \theta(v \to u)\\
        & = \sum_{\underset{v \neq u \in V_{\geq t}}{u\in X}} \theta(v \to u) \\
        & \leq \sum_{\{u, v \} \in E(K_{X \leftrightarrow V_{\geq t}})} [\theta(u \to v) + \theta(v \to u)] \\
        & =  \sum_{\{u, v \} \in E(K_{X \leftrightarrow V_{\geq t}})} [1] =  |K[X \leftrightarrow V_{\geq t}]|.
        \end{split}
    \end{equation}
    On the other side,
    \begin{equation}
        \begin{split}
            \sum_{u \in Y} w_K(u) & = \sum_{\underset{v \in V \setminus \{ u\}}{u \in Y}} \theta(v \to u)  \\
             &\geq  \sum_{\underset{v \neq u \in (V_{\geq t} \cup Y)}{u \in Y}} \theta(v \to u) \\
             & = \sum_{\{u, v \} \in E(K_{Y \leftrightarrow (V_{\geq t}\cup Y)})} [\theta(u \to v) + \theta(v \to u)]\\
             &             = |K[Y \leftrightarrow (V_{\geq t}\cup Y)]|.
        \end{split}
    \end{equation}

    The last point follows by taking $X = V_{\geq t}$ in the first inequality and $Y = V_{\geq t}$ in the second inequality (so that $Y \subseteq V_{<t'}$ for some large $t'$).
\end{proof}

\begin{proof}[Proof of Proposition \ref{balancedload2}] In this proof, we fix a balanced load allocation $\theta$ for the graph $K$.

    \underline{Point 1.} Let $K'$ be a connected component of $K$ with vertex set $\mc{A}$, and assume that for some vertex $u$ in $K'$, $w_{K}(u) \leq 1-\eps$. Then, any neighbour $v$ of $u$ must also verify $w_K(v) \leq 1-\eps$; otherwise, we would have $\theta(u \to v) = 0$, $\theta(v \to u) = 1$ and $w_K(u) \geq 1$. By propagation, all vertices $v$ of $K'$ verify $w_K(v) \leq 1-\eps$; thus,

    \begin{equation}
        |K'| = \sum_{v \in \mc{A}}w_{K'}(v) \leq (1 - \eps)|\mc{A}|.
    \end{equation}
    Since $K'$ is connected, this is only possible if $K'$ is a tree and $|\mc{A}| \leq \frac{1}{\eps}$.

    \underline{Point 2.} Let $v \in V_{\geq k}$. By Proposition \ref{balancedload1}, 

    \begin{equation}
        |K[\{v\} \leftrightarrow V_{\geq k}]| \geq w_K(v) \geq k.
    \end{equation}
    Thus, $V_{\geq k}$ is contained within $\Ker_k(K)$.

    Now, let $V_{ker}$ be the vertex set of $\Ker_k(K)$ and $V_{ker}^- = \{ v \in V_{ker}: w_K(v) < \frac{k}{2} \}$. Again by Proposition \ref{balancedload1}, if $|V_{ker}^-| > 0$,
    \begin{equation}
        \frac{k}{2}|V_{ker}^-| > \sum_{u \in V_{ker}^-} w_K(u) \geq |K[V_{ker}^- \leftrightarrow V_{ker}]| \geq \frac{1}{2} k |V_{ker}^-|
    \end{equation}
    which is impossible. Thus, $|V_{ker}^-| = 0$.
    
    \underline{Point 3.}
    Let $\rho_m' = \max_{v \in V}w_K(v)$; we will show that $\rho_m = \rho_m'$.
    First of all, by Proposition \ref{balancedload1},

    \begin{equation} \label{tempbal2}
        \frac{|K[V_{\geq \rho'_m}]|}{|V_{\geq \rho'_m}|} \geq \frac{1}{|V_{\geq \rho'_m}|}\sum_{v \in V_{\geq \rho'_m}} w_K(v) = \rho'_{m}
    \end{equation}
    so $\rho_m \geq \rho_m'$.

    Furthermore, if $Y \subseteq V$ verifies $\dfrac{|K[Y]|}{|Y|} = \rho_m$, then 

    \begin{equation}
        |Y| \rho_m' \geq \sum_{v \in Y} w_K(v) \geq |K[Y]|
    \end{equation}
    and $\rho_m' \geq \rho_m$.
    The second statement then follows from (\ref{tempbal2}).

\end{proof}

\begin{proof}[Proof of Proposition \ref{balancedload3}]
    The first point is established in \cite[Lemma 2.2]{du2025}.

    For the second point, we will first assume that $G \subseteq H$. Then, by Proposition \ref{balancedload1}, 

\begin{equation}
    \E_{v \sim \Unif(V)}[w_H(v) - w_G(v)] \leq \eps
\end{equation}

and thus, by Markov's inequality, we can construct an $\mc{A}$ verifying $|\mc{A}| \leq \sqrt{\eps} n$ and, for $v \notin \mc{A}$, $|w_G(v) - w_H(v)| \leq \sqrt{\eps}$.

The general case follows by applying this to $(G, G \cup H)$ and $(H, G \cup H)$.
\end{proof}

\section{Proof of Proposition \ref{starsnstripes}}

\begin{proof}
Without loss of generality, we may assume that $K$ is a tree. Let $M > 0$ and define 
\begin{equation}
    V^+ = \{ v \in [n]: \deg v \geq M \}; \qquad V^- = [n] \setminus V^+.
\end{equation}

We will distinguish two cases.

\underline{Case 1:  $\sum_{v \in V^+} \deg v \geq \frac{n}{M}$}. 

Fix an arbitrary root $v \in [n]$, to make $K$ into a rooted tree. The vertices of $K$ thus all have either even depth or odd depth. We assume without loss of generality that

\begin{equation}
    \sum_{v \in V^+ \text{ of even depth}} \deg v \geq \frac{n}{2M}.
\end{equation}

Then, we may pick $\sigma \in \Sym_n$ such that:
\begin{itemize}
    \item if $v \in V^+$ is of even depth then $\sigma$ swaps around the children of $v$, leaving at most one fixed point;
    \item $\sigma$ fixes all other points.
\end{itemize}
Thus, if $\mc{A} = \{ v \in [n] : \text{ the mother of } v \text{ is in } V^+ \text{ and has even depth, and } v \in \Der \sigma \}$, we have

\begin{equation}
    |\mc{A}| \geq n \frac{M-2}{2 M^2} \qquad \text{and} \qquad\min_{v \in \mc{A}} w_{K \land_\sigma K}(v) \geq \frac{M - 1}{M}.
\end{equation}

\underline{Case 2:  $\sum_{v \in V^+} \deg v < \frac{n}{M}$}. 

In this case, $|K[V^-]| \geq \frac{M-1}{M}n $: thus, $K[V^-]$ is a forest and contains at least $\frac{n}{2}$ vertices in components consisting of more than $\frac{M}{2}$ vertices. We may then greedily subsample vertices, to form $V^{--} \subseteq V^-$ such that $K[V^{--}]$ is a forest, where:

\begin{itemize}
    \item all connected components of $K[V^{--}]$ have between $\frac{M}{2}$ and $M$ vertices;
    \item $|V^{--}| \geq \frac{n}{2M}$.
\end{itemize}

In particular, there is a finite number of isomorphism classes for the connected components of $K[V^{--}]$ to belong to (independent from $n$). Thus, if we remove a subset $I$ of size $O(1)$ from $V^{--}$, the resulting graph $K[V^{--} \setminus I]$ has an automorphism $\sigma$ with no fixed points. Extending $\sigma$ to $\Sym_n$ by fixing everything else, we thus obtain

\begin{equation}
    \min_{v \in V^{--} \setminus I} w_{K \land_\sigma K}(v) \geq \frac{\frac{M}{2} - 1}{\frac{M}{2}}.
\end{equation}

 Combining the two cases, we have thus proven that for arbitrary $M$, $SOV(K) \geq \frac{M-2}{M}$; thus, $SOV(K) \geq 1$.
\end{proof}

\section{Additional proofs for Theorem \ref{ersov}}
In this section, we prove Lemmas \ref{tailboundsfixedsigma} and \ref{whatastatement}.

\begin{proof}[Proof of Lemma \ref{tailboundsfixedsigma}]

Let $k \geq 1$ and $\sigma \in \Sym_n$, and fix a $\sigma$-orbit $O \subseteq E[X \leftrightarrow (X \cup \Fix \sigma)]$ of length $k$. Then, by \cite[Appendix A.1]{dingdu2023}, setting $e_O = |(K \land_\sigma K) \cap O|$, for any $\theta > 0$ we have

\begin{equation}
    \E[e^{\theta e_O}] = \mu_1^k + \mu_2^k
\end{equation}
where $(\mu_1, \mu_2)$ are the roots of the polynomial

\begin{equation}
    X^2 - (1 + \frac{\lambda(e^\theta - 1)}{n^\alpha}) X + (e^\theta - 1)(\frac{\lambda}{n^\alpha} - \frac{\lambda^2}{n^{2\alpha}}).
\end{equation}

If we pick $1 \ll e^\theta \ll n^\alpha$, we then have the following identities:

\begin{equation}
    \mu_1 = 1 + \frac{\lambda^2 e^\theta}{n^{2\alpha}} + o(\frac{e^\theta}{n^{2 \alpha}}); \hspace{10pt} \mu_2 = \frac{\lambda e^\theta}{n^\alpha}(1 + o(1)).
\end{equation}

Thus, if $\rho \geq 0$, we may pick $\theta = (\alpha - \frac{1}{k}) \log n$ and so, if $C_k(\sigma) = \frac{1}{k}|\mc{O}_k|$ is the number of orbits inside $\mc{O}_k$,

\begin{equation}
    \begin{split}
        \Proba(e_k(\sigma) \geq \rho n) & \leq e^{-\theta \rho n} \E[e^{\theta e_k(\sigma)}] \\
        & = e^{-\theta \rho n} \left( \mu_1^k + \mu_2^k \right)^{C_k(\sigma)} \\
        & \leq e^{-\theta \rho n} \left(1 + (k \lambda^2 + \lambda^k) n^{-1} + o(n^{-1}) \right)^{C_k(\sigma)} \\
        & = \exp\left(-(\alpha - \frac{1}{k}) \rho n \log n + O(n)\right)
    \end{split}
\end{equation}
uniformly in $k \leq L$, since $C_k(\sigma) \leq n^2$. Similarly, picking $\theta = (2\alpha - 1) \log n$, we obtain that

\begin{equation}
    \begin{split}
        \Proba(e_{>L}(\sigma) \geq \rho n)  & \leq e^{-\theta \rho n} \E[e^{\theta e_{>L}(\sigma)}] \\
        & = e^{-\theta \rho n} \prod_{k \geq L+1} (\mu_1^k +\mu_2^k)^{C_k(\sigma)} \\
        & \leq e^{-\theta \rho n} \prod_{k \geq L+1}(\mu_1^{L+1} +\mu_2^{L+1})^{\frac{k C_k(\sigma)}{L+1}} \\
        & \leq e^{-\theta \rho n} (\mu_1^{L+1} +\mu_2^{L+1})^{n^2} \\
        & = e^{-(2\alpha-1)\rho n\log n} (1 + O(n^{-1}))^{n^2} \\
        & = \exp\left(-(2 \alpha - 1) \rho n \log n + O(n)\right)
    \end{split}
\end{equation}

Furthermore, picking $\theta = \alpha \log n - \log \log n$, we obtain that

\begin{equation}
    \begin{split}
        \Proba(e_s(\sigma) \geq \rho n) & \leq e^{-\theta \rho n} \E[e^{\theta e_s(\sigma)}] \\
        & =  e^{-\theta \rho n} \prod_{O \subseteq \mc{O}_s} (\mu_1^{|O|} + \mu_2^{|O|}) \\
        & \leq e^{-\alpha \rho n \log n + O(n \log \log n)} (\mu_1 + \mu_2)^n \\
        & = \exp \left( - \alpha \rho n \log n + o(n \log n)\right).
    \end{split}
\end{equation}
since $|\mc{O}_s| \leq n$.

Finally, we need to handle the chains. If $C$ is a chain of length $k$, then, again by \cite[Appendix A.1]{dingdu2023}, for any $\theta$ such that $1 \ll e^\theta \ll n^\alpha$,

\begin{equation}
    \E[e^{\theta e_C}] \leq \mu_1^k + e^\theta \frac{\lambda^3}{n^{3 \alpha}}(1 + o(1)) \mu_2^k 
\end{equation}

so that, picking $\theta = (2 \alpha - 1) \log n$,

\begin{equation}
    \begin{split}
        \Proba(e_{c}(\sigma) \geq \rho n)  & \leq e^{-\theta \rho n} \E[e^{\theta e_{c}(\sigma)}] \\
        & = e^{-\theta \rho n} \prod_{C}  \E[e^{\theta e_{C}(\sigma)}] \\
        & = e^{-(2\alpha-1)\rho n\log n} (1 + O(n^{-1}))^{n^2} \\
        & = \exp\left(-(2 \alpha - 1) \rho n \log n + O(n)\right)
    \end{split}
\end{equation}

since the number of chains is at most $n^2$.

\end{proof}

\begin{proof}[Proof of Lemma \ref{whatastatement}]
    Set $\kappa_2 = \frac{|X_2|}{n}, \ldots, \kappa_c = \frac{|X_c|}{n}$. Then, for any $2 \leq k \leq L$, the values $e_2(\sigma, X), \ldots, e_k(\sigma, X)$ only depend on $\sigma$ via $((\sigma^{-1})|_{X_i})_{2 \leq i \leq k}$; and the complete set of values $e_2(\sigma, X), \ldots, e_L(\sigma, X), e_{>L}(\sigma, X), e_c(\sigma, X), e_s(\sigma, X)$ only depends on $\sigma$ via $(\sigma^{-1})|_X$. In particular, since the $(e_i(\sigma, X))$ are mutually independent, we may apply Lemma \ref{tailboundsfixedsigma} and obtain that

\begin{equation}
    \begin{split}
        \Proba &\left(\exists \sigma \in \Sym_n: X \subseteq \Der \sigma,\begin{cases}
            \mc{C}_2(\sigma, X) = X_2 \\
            \mc{C}_3(\sigma, X) = X_3 \\
            \hspace{29pt}\vdots \\
            \mc{C}_k(\sigma, X) = X_k
        \end{cases}, \begin{cases}
            e_2(\sigma, X) \geq \rho_2 n \\
            e_3(\sigma, X) \geq \rho_3 n \\
            \hspace{29pt}\vdots \\
            e_k(\sigma, X) \geq \rho_k n \\
        \end{cases} \right) \\
        & \leq \left(\prod_{i=2}^k \frac{|X_i|!}{\left(\frac{1}{i} |X_i|\right)!}\right) \exp \left( - (\alpha - \frac{1}{2}) \rho_2 n \log n- \ldots - (\alpha - \frac{1}{k}) \rho_k n\log n + o(n \log n)\right) \\
        & \leq \exp \left( \sum_{i=2}^k [(1 - \frac{1}{i}) \kappa_i - (\alpha - \frac{1}{i})\rho_i] n \log n + o(n \log n)\right)
    \end{split}
\end{equation}

    and, similarly, 

    \begin{equation}
        \begin{split}
            \Proba & \left(\exists \sigma \in \Sym_n:  X \subseteq \Der \sigma,\begin{cases}
            \mc{C}_2(\sigma, X) = X_2 \\
            \mc{C}_3(\sigma, X) = X_3 \\
            \hspace{29pt}\vdots \\
            \mc{C}_c(\sigma, X) = X_c
        \end{cases}, \begin{cases}
            e_2(\sigma, X) \geq \rho_2 n \\
            e_3(\sigma, X) \geq \rho_3 n \\
            \hspace{29pt}\vdots \\
            e_c(\sigma, X) \geq \rho_c n \\
            e_s(\sigma, X) \geq \rho_s n
        \end{cases} \right) \\
        & \leq         \left(\prod_{k=2}^L \frac{|X_k|!}{\left(\frac{1}{k} |X_k|\right)!}\right)\frac{|X_{>L}|!}{\left(\frac{1}{L+1} |X_{>L}|\right)!} n^{|X_c|} \exp \left( -\left[ \sum_{i=2}^L (\alpha - \frac{1}{i}) \rho_i + (2 \alpha - 1) (\rho_{>L} + \rho_c) + \alpha \rho_s \right]n\log n+o(n \log n) \right) \\
        & \leq \exp \left( \sum_{i=2}^L [(1 - \frac{1}{i}) \kappa_i - (\alpha - \frac{1}{i})\rho_i]n \log n + [\kappa_{res} - (2 \alpha - 1)\rho_{res}]n\log n + o(n \log n) \right)
        \end{split}
    \end{equation}
    where $\kappa_{res} = \kappa_{>L} + \kappa_{c}$ and $\rho_{res} = \rho_{>L} + \rho_c + \rho_s$. To conclude the proof, it is thus sufficient to show that one of the following holds:

    \begin{itemize}
        \item for some $k \in \llb 2, L \rrb$,  $S_k \defeq \dsum_{i=2}^k [(1 - \frac{1}{i}) \kappa_i - (\alpha - \frac{1}{i})\rho_i] \leq -\delta < 0$;
        \item or $S_{tot} \defeq \dsum_{i=2}^L [(1 - \frac{1}{i}) \kappa_i - (\alpha - \frac{1}{i})\rho_i] + [\kappa_{res} - (2 \alpha - 1)\rho_{res}] \leq -\delta < 0$.
    \end{itemize}

    To prove this, we will use the following trick. For $2 \leq i \leq L$, set

    \begin{equation}
        w_i = \frac{1}{\alpha - \frac{1}{i}}; w_{res} = \frac{1}{2 \alpha - 1}.
    \end{equation}
    such that $(w_2, \ldots, w_L, w_{res})$ is a decreasing sequence.
    
    Then, summing by parts,

    \begin{equation}
        \begin{split}
            \sum_{i=2}^{L-1} & (w_i - w_{i+1})S_i + (w_L - w_{res}) S_L + w_{res}S_{tot} \\
            &= \sum_{i=2}^L w_i[(1 - \frac{1}{i}) \kappa_i - (\alpha - \frac{1}{i})\rho_i] + w_{res}[\kappa_{res} - (2 \alpha - 1)\rho_{res}] \\
            & \leq \sum_{i = 2}^L [\frac{1}{2 \alpha -1} \kappa_i - \rho_i] + \frac{1}{2 \alpha -1}\kappa_{res} - \rho_{res} \\
            & \leq \frac{1}{2 \alpha - 1}\kappa - (\frac{1}{2 \alpha - 1} + \eta)\kappa = -\eta \kappa.
        \end{split}
    \end{equation}

    This means that, for some $i \in \llb 2, L \rrb \cup \{ \text{tot}\}$,

    \begin{equation}
        S_i \leq -\frac{\kappa \eta}{w_2}
    \end{equation}
    concluding the proof.
\end{proof}

\section{Additional proofs for Theorem \ref{impossible}} \label{impossibleproofs}
In this section, we prove Lemmas \ref{bayesian}, \ref{ppost}, and \ref{negdef}.

\begin{proof}[Proof of Lemma \ref{bayesian}]
    Let $B = B(\pi_0, x - \eta)$ be any ball of radius less that $x$ within $\Sym_n$. We know that

    \begin{equation}
        \E[\Proba(\Neg|G, H, (p_e))] = \Proba(\Neg) = o(1)
    \end{equation}
    
    so $\Proba(\Neg|G, H, (p_e)) = o_\Proba(1)$. Thus:

    \begin{equation}
        \begin{split}
            \Proba(\pi^* \in B | G, H, (p_e)) & \leq \Proba(\pi^* \in\Neg|G, H, (p_e)) + \sum_{\pi \in B \setminus \Neg} \Proba(\pi^* = \pi|G, H, (p_e)) \\
            &\leq o_\Proba(1) + e^{-\gamma n \log n} \sum_{\pi \in B \setminus \Neg}  \Proba(\pi^* \in F(\pi)|G, H, (p_e)) \\
            & \leq o_{\Proba}(1) + e^{-\gamma n \log n} \sum_{\pi' \in \Sym_n} \sum_{\pi \text{ s.t. }\pi' \in F(\pi)} \Proba(\pi^* =\pi'|G, H, (p_e)) \\
        & \leq o_{\Proba}(1) + e^{(-\gamma + o(1))n \log n} \sum_{\pi' \in \Sym_n}\Proba(\pi^* =\pi'|G, H, (p_e)) \\
        & = o_{\Proba}(1) + e^{(-\gamma + o(1))n \log n} \Proba(\pi^* \in\Sym_n|G, H, (p_e)) \\
        & = o_{\Proba}(1).
        \end{split}
    \end{equation}
    The previous argument can also easily be seen to be uniform in $\pi_0 \in \Sym_n$.
    By standard Bayesian decision theory arguments (see \cite[Section 2.2]{vassauxmassoulie2026} for instance), this shows that any estimator $\hat{\pi}$ verifies $\Proba(d(\hat{\pi}, \pi^*) \leq x- \eta) = o(1)$, concluding the proof.
\end{proof}

\begin{proof}[Proof of Lemma \ref{ppost}]
    
    Let $\pi \in \Sym_n$ and $G_0,H_0 \in \mathbb G_n$. We first compute the
joint probability of the event $\{\pi^*=\pi,\,G=G_0,\,H=H_0\}$, conditionally
on $(p_e)_{e\in E}$. We have:
\begin{align}
& \mathbb P(\pi^*=\pi,\,G=G_0,\,H=H_0\mid (p_e)) \notag\\
&\qquad =
\frac{1}{n!}\prod_{e\notin G_0\vee_\pi H_0}
\bigl(1-p_e(2s-s^2)\bigr)
\Bigl(\prod_{e\in G_0\vee_\pi H_0}
p_e \Bigr) s^{|\pi(G_0)| + |H_0|}(1-s)^{|\pi(G_0) \Delta H_0|}
.
\end{align}
Note that $\prod_{e\in E}
\bigl(1-p_e(2s-s^2)\bigr)$ and $s^{|\pi(G_0)| + |H_0|}$ do not depend upon $\pi$. Thus, uniformly in $\pi$,
\begin{align}
\mathbb P(\pi^*=\pi,& \, G=G_0,\,H=H_0\mid (p_e)) \\
&\propto
\exp\left(
\sum_{e\in G_0\vee_\pi H_0} \log \left( \frac{p_e}{1 - p_e(2s - s^2)} \right) + |G_0 \Delta_{\pi} H_0| \log (1-s)
\right) \notag\\
& \defeq \exp\left(
\sum_{e\in G_0\vee_\pi H_0} l_e + |G_0 \Delta_{\pi} H_0| \log (1-s)
\right)
\end{align}

Set $\eps > 0$. Let $\Neg_1 \subseteq \Sym_n$ be the subset of $\pi \in \Sym_n$ such that

\begin{equation}
    \# \{ e \in G \lor_{\pi} H: p_e \geq \frac{\log n}{n} \} \geq \eps n
\end{equation}
By our hypotheses, $\Proba(\pi^* \in \Neg_1) = o(1)$.

Now, let $\pi \in \Sym_n \setminus \Neg_1$; and let $\sigma$ be an automorphism of $I(\pi)$.
Then 
\begin{equation}
I(\pi)\subseteq \sigma\pi(G)\cap H = I(\sigma\pi).
\end{equation}
so that $|U(\pi)| \geq |U(\sigma \pi)|$. 

Furthermore, if $e \in E$: 
\begin{itemize}
    \item if $p_e \leq \frac{\log n}{n}$, we also know that $p_e \geq \frac{1}{n (\log n)^r}$ so that $l_e = - \log n + O(\log \log n)$;
    \item otherwise, we still know that $\frac{1}{n (\log n)^r} \leq p_e \leq 1-n^{-3}$ and $-\log n + O(\log \log n) \leq l_e \leq 3 \log n$. 
\end{itemize}
Thus, splitting up the two cases: 
\begin{equation}
    \begin{split}
        \sum_{e \in U(\pi)} l_e  & \leq 3 \eps n \log n - |U(\pi)| \log n (1  +O(\frac{\log \log n}{\log n})) \\
        & = -(|U(\pi)| - 3 \eps n + o(n)) \log n. \\
    \end{split}
\end{equation}

As a result,
\begin{equation}
    \begin{split}
         \sum_{e \in U(\pi)} l_e  & \leq -(|U(\sigma \pi)| - 3 \eps n + o(n)) \log n \\
         & \leq \sum_{e \in U(\sigma \pi)} l_e + (3 \eps n + o(n)) \log n.
    \end{split}
\end{equation}

This means that for arbitrarily small $\eps > 0$, 

\begin{equation}
    \Proba(\pi^* = \pi | G, H, (p_e)) \leq \Proba(\pi^* = \sigma \pi | G, H, (p_e))e^{(3 \eps n + o(n)) \log n}
\end{equation}

concluding the proof.
\end{proof}

\begin{proof}[Proof of Lemma \ref{negdef}]
    Let us first prove that, with high probability, the permutation $\pi^*$ verifies the property detailed in the lemma, which we denote by $(\mc{P}(\pi^*))$. Indeed, by definition of $x_{sparse}$, there exists $c_\eps \goesto{\eps \to 0} 0$ such that $|V_{\leq(1-\eps)}(\pi^*)| = |V_{\leq(1-\eps)}| \geq x_{sparse} n(1-c)$ whp; and, by Proposition \ref{balancedload2}, $I[V_{\leq(1-\eps)}]$ is a union of connected components of size at most $\frac{1}{\eps}$. Thus, given a subset $A$ verifying the conditions of the lemma, we may partition $A$ into subsets $A_1, \ldots, A_p$ such that no edges of $I$ link $A_i$ to $A_j$ for $i \neq j$; and all connected components of $I_{A_i}$ are isomorphic. Since $p$ is independent of $n$ (at most equal to $2^{\frac{1}{\eps^2}}$), for some $\gamma > 0$, there are at least $\exp(\gamma n \log n)$ automorphisms $(\sigma_i)$ of $A$. Extending these automorphisms to $V_{sparse}^\eps$ by fixing the points of $A^c$ shows that $\pi^*$ indeed verifies the required property.

This means that, if we set $\Neg_2 = \{ \pi \in \Sym_n: \mc{P}(\pi) \text{ is false }\}$, $\Proba(\pi^* \in \Neg_2) = o(1)$, concluding the proof.
\end{proof}

\section{Proofs of results from Section \ref{corollaries}} \label{corollaryproofs}

In this section, we provide the extra elements necessary in order to deduce the corollaries from Section \ref{corollaries}. \\

\underline{Corollary \ref{cores}.}
This follows from Theorem \ref{possible} and Proposition \ref{balancedload2}.\\

\underline{Corollary \ref{giant}.}
Our job is to show that under the conditions of the corollary, there exists $\gamma, \delta > 0$ such that  $|V_{\geq(1 + \gamma)}| \geq \delta n$ whp; applying Theorem \ref{possible} then shows the corollary.  

We will fix $\gamma, \delta = \frac{\eps}{2}, \frac{\eta}{2}$. Assume by contradiction that $|V_{\geq(1 + \gamma)}| < \delta n$ with non-negligible probability. Then, by hypothesis,

\begin{equation}
    \mc{G}[V_{\leq(1 +\gamma)}] \geq (1 + \eps) |V_\mc{G}| \text{ with non-negligible probability},
\end{equation}

but, by Proposition \ref{balancedload1},
\begin{equation}
    \mc{G}[V_{\leq(1 +\gamma)}] \leq (1 + \gamma) |V_\mc{G}|,
\end{equation}

which is a contradiction. \\

\underline{Corollary \ref{compare}.}
The first point follows from Corollary \ref{cores} and Theorem \ref{ersov}. For the second point, simply note that by \cite[Lemma IV.1]{raczsridhar2023}, the $k$-core occupies a proportion $1 - o_\Proba(1)$ of all vertices, concluding the proof.  \\

\underline{Corollary \ref{chunglu}.}
Note that under the conditions of the corollary, $I$ is a Chung-Lu graph, with law $\mu = s^2 \nu$. Let us begin by proving that
\begin{equation} \label{notactuallytrivial}
    \mu \text{ is weakly inhomogeneous} \Longleftrightarrow \E[\mu] < +\infty.
\end{equation}

On one hand, assume that $\E[\mu] < +\infty$. Set $S = \dfrac{1}{n}\dsum_{w \in V} d_w$; then, by a standard Chernoff bound, there exists $c > 0$ such that $\Proba(S \leq \frac{1}{2} \overline{d}) \leq e^{-c n}$. Then, for any $M > 0$, picking $\mathbf{e} = \{ u, v \} \in $ $ V$ uniformly at random, 

\begin{equation}
    \begin{split}
        \E[np_{\mathbf{e}} \1_{p_{\mathbf{e}} \geq \frac{M}{n}}] & \leq \E\left[\left( \frac{d_u d_v}{S} \1_{S \geq \frac{\overline{d}}{2}} + n \1_{S < \frac{\overline{d}}{2}}\right) \1_{p_{\mathbf{e}} \geq \frac{M}{n}} \right]  \\
        & \leq \frac{2}{\overline{d}} \E[d_u d_v\1_{d_u d_v \geq \frac{\overline{d} M}{2} }] + n e^{-cn}  \\
        \goesto{n \to +\infty, M \to +\infty} 0,
    \end{split}
\end{equation}
since $d_u d_v \in L^1$.

On the other hand, if $\E[\mu] = +\infty$, the expected number of edges in $I$ is $\omega(n)$, and thus $I$ cannot be weakly inhomogeneous.
\\

Having now proven (\ref{notactuallytrivial}), we move on proving our corollary. We will rely upon the following statement, which follows from \cite{vanderhofstad2024}.

\begin{theorem} \label{giantexistencechunglu}
    Let $K$ be a Chung-Lu graph with associated degree law $\mu \in L^1$, and set $d^*(K) = \frac{\E[D^2]}{\E[D]}$ where $D \sim \mu$.
    \begin{itemize}
         \item Assume that $d^* > 1$. Then there exist constants
    $c_K,\delta_K>0$ such that, with high probability, $K$ has a connected
    component $\mc{G}$, with vertex set $V_\mc{G}$, satisfying
    $|V_\mc{G}| \geq c_K n$ and
    \begin{equation}
       |K[V_\mc{G}]| \geq (1+\delta_K)|V_\mc{G}|.
        \end{equation}
        Moreover, the degrees inside $\mc{G}$ are uniformly integrable in the
        following sense: for every $\zeta>0$, there exists $\eta>0$ such that,
        with high probability,
        \begin{equation}
         \sup_{\substack{N \subseteq V_\mc{G}\\ |N|\leq \eta n}}
         \sum_{v\in N}\deg_K(v)
         \leq \zeta |V_\mc{G}|.
        \end{equation}
        \item Assume that $d^* < 1$. If we pick $v \in V$ uniformly at random, then the connected component $\mc{T}_v$ of $v$ is a tree with high probability; furthermore, there exists a random variable $\mc{T}$ on the set of finite-size trees such that
        \begin{equation}
            \mc{T}_v \overset{\text{(law)}}{\goesto{n \to +\infty}} \mc{T}.
        \end{equation}
    \end{itemize}
\end{theorem}

In the first case, we will use Corollary \ref{giant} to conclude. Pick $\zeta = \frac{\delta_I}{2}$ (as in the first item of the theorem). Then, if $N \subseteq V_\mc{G}$ with size at most $\eta(\zeta) \cdot n$,

\begin{equation}
    |I[V_\mc{G} \setminus N]| = |I[\mc{G}]| - \sum_{v \in N} \deg_I(v) \geq (1 + \frac{\delta_K}{2} - o_\Proba(1)) |V_{\mc{G}}|.
\end{equation}
This verifies the conditions of Corollary \ref{giant}.

In the second case, the tree-convergence statement implies that 

\begin{equation}
    \lim_{\eps \to 0} \operatorname*{pliminf}_{n \to +\infty} \frac{|V_{\leq(1 - \eps)}|}{n} = 1
\end{equation}

by Proposition \ref{balancedload2}. If $\nu$ is supported over $[\eps, +\infty)$ for some $\eps > 0$ then applying Theorem \ref{impossible} concludes.

\underline{Corollary \ref{sbm}.} 

By \cite{vanderhofstad2024}, Theorem \ref{giantexistencechunglu} holds as-is for the stochastic block model if we replace $d^*$ by $\lambda^*$. Thus, the exact same argument allows us to deduce Corollary \ref{sbm}.

\section{Regarding the correctly aligned subsets} \label{discussion}

In this section, we discuss the discrepancy between the sets $V_{\geq(\rho + \eps)}$ and $\widehat{V_{dense}}$ from Theorem \ref{possible}. We will prove that the two are roughly the same in the weakly inhomogeneous case, before giving an example where verifying $|V_{\geq(\rho + \eps)}|, |\widehat{V_{dense}}| = \Theta_\Proba(n)$ but $|\widehat{V_{dense}} \cap V_{\geq(\rho + \eps)}| = o_\Proba(n)$.

Let us begin with our positive result.

\begin{proposition} \label{posrecover}
    Assume that $(G, H)$ is a weakly inhomogeneous system. Then, for any $0 < \eps' < \eps $,
    
    \begin{equation}
        \widehat{V_{dense}}(\eps) \subseteq V_{\geq (\rho + \eps')} \qquad \text{(up to }o_\Proba(n) \text{ vertices)}.
    \end{equation}
\end{proposition}

As a result, if $\ds\lim_{\eps' \to \eps} \operatorname*{plimsup}_{n \to +\infty} \frac{|V_{\geq (\rho + \eps')} \setminus V_{\geq (\rho + \eps)}|}{n} = 0$ (which, in most models of interest, will be the case for most $\eps > 0$), then $|V_{\geq (\rho + \eps)} \setminus \widehat{V_{dense}}(\eps) | = o_\Proba(n)$. Since $|V_{\geq (\rho + \eps)}| \leq |\widehat{V_{dense}}(\eps)|$, this thus implies that, in this case,

\begin{equation}
    \widehat{V_{dense}}(\eps) = V_{\geq (\rho + \eps)} \qquad \text{(up to }o_\Proba(n) \text{ vertices)}.
\end{equation}

\begin{proof}
    We will denote $\widehat{V_{dense}} = \widehat{V_{dense}}(\eps)$. Fix $\eps' < \eps$, and set 

    \begin{equation}
        N = \{ v \in \widehat{V_{dense}} : \hat{\pi}(\pi^*)^{-1}(v) \neq v \}; \qquad M = \widehat{V_{dense}} \setminus (V_{\geq (\rho + \eps')} \cup N).
    \end{equation}

    Then,
    \begin{equation}
        \begin{split}
            (\rho + \eps )|M| & \leq \sum_{v \in M} w_{I(\hat{\pi})}(v) \\
            & \leq |I(\hat{\pi})[M \leftrightarrow \widehat{V_{dense}}]| \qquad \text{(by Proposition \ref{balancedload1})} \\
            & = |I(\hat{\pi})[M \leftrightarrow (\widehat{V_{dense}} \setminus N)]| + |I(\hat{\pi})[M \leftrightarrow N]|.
        \end{split}
    \end{equation}
    By Theorem \ref{possible}, $|N| = o_\Proba(n)$; thus, since the system is weakly inhomogeneous, $|I(\hat{\pi})[M \leftrightarrow N]| = o_\Proba(n)$. Furthermore, since $\hat{\pi} = \pi^*$ over $M$ and over $ (\widehat{V_{dense}} \setminus N)$, we may write

    \begin{equation}
        \begin{split}
            |I(\hat{\pi})[M \leftrightarrow (\widehat{V_{dense}} \setminus N)]|& + |I(\hat{\pi})[M \leftrightarrow N]|  = |I(\pi^*)[M \leftrightarrow (\widehat{V_{dense}} \setminus N)]| + o_\Proba(n) \\
            & \leq |I(\pi^*)[M \leftrightarrow (M \cup V_{\geq(\rho + \eps')} )]| + o_\Proba(n) \\
            & \leq \sum_{v \in M} w_I(v) + o_\Proba(n)\qquad \text{(by Proposition \ref{balancedload1})} \\
            & \leq (\rho + \eps')|M| + o_\Proba(n).
        \end{split}
    \end{equation}

    This necessarily implies that $|M| = o_\Proba(n)$, concluding the proof.
    
    \end{proof}

As promised. we will now build an example of a system where $|V_{\geq(\rho + \eps)}|, |\widehat{V_{dense}}| = \Theta_\Proba(n)$ but $|\widehat{V_{dense}} \cap V_{\geq(\rho + \eps)}| = o_\Proba(n)$. The following construction was initially proposed by a large language model; the author has made significant technical changes and verified all steps.

Fix an odd integer $d\geq 5$ and choose $\beta\in(\frac{4}{d},1)$. Set
\begin{equation}
    q=q_n=\lfloor n^\beta\rfloor.
\end{equation}
Partition $V=[n]$ into four disjoint sets
\begin{equation}
    V=A\sqcup B\sqcup C\sqcup D
\end{equation}
with
\begin{equation}
    |C|=|D|=q,\qquad |A|\sim \frac n3,
    \qquad |B|\sim \frac{2n}{3}.
\end{equation}
Let $\iota:C\to D$ be a bijection. For each $x\in A \sqcup B$, choose a
$d$-element subset
\begin{equation}
    S_x\subseteq C.
\end{equation}
Equivalently, we are choosing two random left-$d$-regular bipartite
graphs, one between $A$ and $C$, and one between $B$ and $C$.

We shall choose the families $(S_x)_{x\in A \sqcup B}$ at
random, independently and uniformly among all $d$-subsets of $C$. Then, setting 

\begin{equation}
    K_A
    =
    \bigl\{\{x,c\}:x\in A,\ c\in S_x\bigr\},
\end{equation}
and
\begin{equation}
    K_B^C
    =
    \bigl\{\{y,c\}:y\in B,\ c\in S_y\bigr\},  \qquad  K_B^D
    =
    \bigl\{\{y,\iota(c)\}:y\in B,\ c\in S_y\bigr\},
\end{equation}
we will define
\begin{equation}
    G=K_A\cup K_B^C,
    \qquad
    H=K_A\cup K_B^D.
\end{equation}
(We simplify notations by assuming that $\pi^* = \Id$; replacing $G$ by $(\pi^*)^{-1}(G)$ gives the general construction.)
Then the true intersection and union graphs are 
\begin{equation}
    I=G\cap H=K_A; \qquad U=G\cup H=K_A\cup K_B^C\cup K_B^D.
\end{equation}

\begin{figure}
    \centering
    \includegraphics[width=\linewidth]{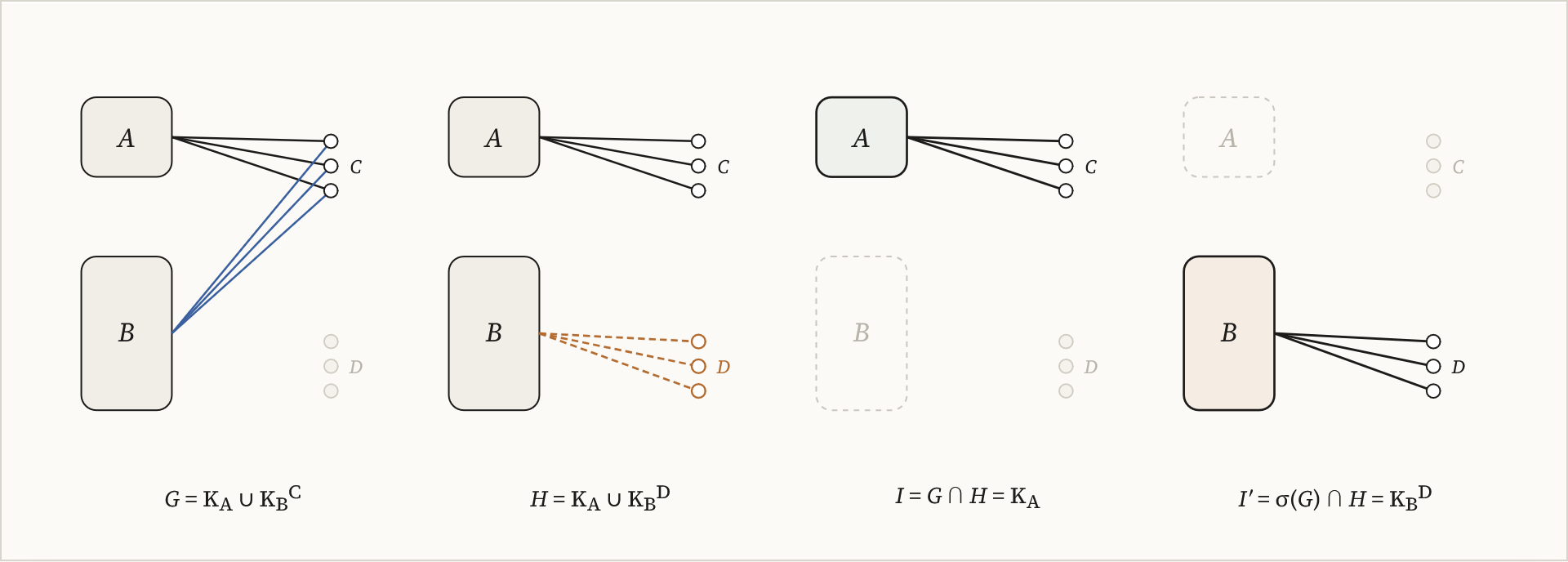}
    \caption{Schematic representation of the counterexample, under the convention $\pi^* = \Id$. The sets $C, D$ have size $o(n)$ but have $\Theta(n)$ edges with at least one endpoint inside them; as per the proof of Proposition \ref{posrecover}, this is necessary.\\}
    \label{cexfigure}
\end{figure}

We will prove that, setting $t = d - \frac{1}{2}$,

\begin{itemize}
    \item $SOV(U) < t$;
    \item $V_{\geq t}(I) = A$ (up to $o_\Proba(n)$ vertices);
    \item $V_{dense} = B$ (up to $o_\Proba(n)$ vertices).
\end{itemize}

\medskip

To show this, we will need the following facts about the model we just constructed.
\begin{lemma}
The following hold with high probability.

\begin{enumerate}
    \item For every $P, Q\subseteq C$, set
    \begin{equation}
        c_A(P) = \frac{1}{n}\#\{x\in A:|S_x \cap P| = d\}, \qquad  c_B(Q) = \frac{1}{n}\#\{y\in B:|S_y \cap Q| = d\}. 
    \end{equation}
    Then, if $P, Q$ are disjoint, $c_A(P) + c_B(Q) \leq \frac{2}{3} + o_\Proba(1)$ (uniformly in $P, Q \subseteq C$), with equality iff $|Q \Delta C| = o_\Proba(q)$.
    \item Uniformly over every permutation $\sigma\in\Sym_n$,
    \begin{equation} \label{sovd-1}
        \#\left\{
            v\in (A\sqcup B)\cap \Der\sigma:
            \deg_{U\land_\sigma U}(v)\geq d
        \right\}
        =o(n).
    \end{equation}
\end{enumerate}
\end{lemma}
\begin{proof}

We begin with the first point. By standard concentration inequalities, uniformly in $P, Q \subseteq C$,
\begin{equation}
    c_A(P) = \frac{1}{3}p(P) + o_\Proba(1); \qquad c_B(Q) = \frac{2}{3} p(Q) + o_\Proba(1),
\end{equation}
where $p(P) \defeq \Proba(|S \cap P| =d)$ for $S$ chosen uniformly at random among $d$-element subsets of $C$. In particular, since $p(P) + p(Q) \leq 1$,

\begin{equation}
    c_A(P) + c_B(Q) = \frac{1}{3}p(P) + \frac{2}{3} p(Q) + o_\Proba(1) \leq \frac{2}{3} + o_\Proba(1),
\end{equation}
with equality iff $|Q \Delta C| = o_\Proba(q)$.

We now justify the second point. 

For $v\in A\sqcup B$, write $N_U(v)$ for its neighbourhood in $U$ inside
the hub set $C\sqcup D$, which is either $S_v$ or $S_v \cup \iota(S_v)$.
In particular,
\begin{equation}
    |N_U(x)|=d\quad(x\in A),
    \qquad
    |N_U(y)|=2d\quad(y\in B).
\end{equation}

Now, fix $\sigma \in \Sym_n$, and  pick $v \in (A \sqcup B) \cap \sigma^{-1}(A \sqcup B)$ uniformly at random. Then: 

\begin{equation} \begin{split}
\Proba & \left( v \in \Der \sigma \text{ and }\deg_{U\land_\sigma U}(v)\geq d \right) \\
    & = \Proba\left(v \in \Der \sigma \text{ and } |N_U(\sigma(v)) \cap \sigma(N_U(v))| \geq d \right) \leq C_d q^{-\frac{d}{2}}
\end{split}\end{equation}
where $C_d$ is a constant which only depends upon $d$. Thus, for any $\delta > 0$, by independence,

\begin{equation}
    \begin{split}
        \Proba &\left( \#\left\{
            v\in (A\sqcup B)\cap \Der\sigma:
            \deg_{U\land_\sigma U}(v)\geq d
        \right\} \geq \delta n\right) \\
        & \leq \Proba \left( \#\left\{
            v\in (A\sqcup B)\cap \sigma^{-1}(A \sqcup B) \cap  \Der\sigma:
            \deg_{U\land_\sigma U}(v)\geq d
        \right\} \geq \delta n\right) \\
        & \leq \Proba \left(\exists \text{ pairwise disjoint } \begin{cases}(v_1, v_1') \\ \ldots \\ (v_{\lceil \frac{\delta n}{2} \rceil}, v_{\lceil  \frac{\delta n}{2} \rceil}') \end{cases} \subseteq  (A\sqcup B)\cap \Der\sigma : \right. \\
        & \left.\forall 1 \leq i \leq \lceil  \frac{\delta n}{2} \rceil, |N_U(v_i') \cap \sigma(N_U(v_i))| > d \right) \\
        & \leq q^{- \frac{\delta n}{2} \frac{d}{2}} = \exp(- \beta \delta \frac{d}{4} n \log n)
    \end{split}
\end{equation}
We may then union bound over:

\begin{itemize}
    \item all possible choices of $(v_i, v_i')$ (there are at most $e^{ \delta n \log n}$ options);
    \item all possible choices of $\sigma|_{C \sqcup D}$ (there are $e^{o(n \log n)}$ options)
\end{itemize}
to obtain that 

\begin{equation}
\begin{split}
     \Proba & \left(\exists \sigma \in \Sym_n,  \#\left\{
            v\in (A\sqcup B)\cap \Der\sigma:
            \deg_{U\land_\sigma U}(v)\geq d
        \right\} \geq \delta n\right) \\
        &\leq e^{( 1 -\beta \frac{d}{4} + o(1)) \delta n \log n} = o(1),
        \end{split}
\end{equation}
concluding the proof.
\end{proof}
\medskip

We now return to the main proof. 

First, note that by (\ref{sovd-1}),
\begin{equation}
    SOV(U) \leq (d-1) < t
\end{equation}
since the balanced load of a vertex is bounded by its degree.

Next, we claim that
\begin{equation} \label{101}
    V_{\geq t}(I)=A \qquad \text{(up to } o_\Proba(n) \text{ vertices)}.
\end{equation}
Indeed, since $I=K_A$, all non-isolated vertices lie in $A \sqcup C$.
Furthermore, by Propositions \ref{balancedload1} and \ref{balancedload2}:
\begin{equation}
    \sum_{v \in A \sqcup C} w_I(v) = d|A| \qquad \text{and} \qquad \max_{v \in A \sqcup C}w_I(v) = \max_{X \subseteq A \sqcup C} \frac{|I[X]|}{|X|} \leq d,
\end{equation}

so that, if we pick $v \in A \sqcup C$ uniformly at random, $w_I(v) = d - o_\Proba(1)$. 
This shows (\ref{101}).

Finally, we are tasked with showing that 
\begin{equation}
    V_{dense} = B  \qquad \text{(up to } o_\Proba(n) \text{ vertices)}.
\end{equation}
Let $\tau\in\Sym_n$ be the permutation defined by
\begin{equation}
    \tau(v)= \begin{cases}
        \iota(v) & \text{if } v \in C \\
        \iota^{-1}(v) & \text{if } v \in D \\
        v & \text{otherwise.}
    \end{cases}
\end{equation}
Then
\begin{equation}
    I(\tau)=\tau(G)\cap H=K_B^D.
\end{equation}

Thus, by the same argument as for (\ref{101}),
\begin{equation} \label{taucool}
    V_{\geq t}(I(\tau))=B  \qquad \text{(up to } o_\Proba(n) \text{ vertices)}.
\end{equation}

Now, let $\sigma\in\Sym_n$ be arbitrary. Since $I(\sigma) \subseteq H$, any vertex has degree at most $d$ inside $I(\sigma)$. We then define

\begin{equation}
    P_\sigma = \{c \in C: \sigma(c) \in C \}; \qquad Q_\sigma = \{c \in C: \sigma(c) \in D \}.
\end{equation}
By the second part of the good incidence event, the number of vertices in $A$ having degree
$d$ in $I(\sigma)$ is at most
\begin{equation}
    \#\{x\in A:S_x\subseteq P_\sigma\}+o_\Proba(n) = n c_{A}(P_\sigma) +o_\Proba(n),
\end{equation}
and the number of vertices in $B$ having degree $d$ in $I(\sigma)$ is at
most
\begin{equation}
    \#\{y\in B:S_y\subseteq Q_\sigma\}+o_\Proba(n) = n c_{B}(Q_\sigma) +o_\Proba(n).
\end{equation}
Using the first part of the good incidence event, the total number of vertices with degree $d$ in $I(\sigma)$ is at most equal to $\frac{2}{3}n + o_\Proba(n)$, and we can only have equality if $|Q_{\sigma} \Delta C| = o_\Proba(n)$. In particular, 
\begin{equation}
    V_{\geq t}(I(\widehat \pi)) \subseteq B \qquad \text{(up to } o_\Proba(n) \text{ vertices).}
\end{equation}
and, by (\ref{taucool}), 

\begin{equation}
        V_{\geq t}(I(\widehat \pi)) = B \qquad \text{(up to } o_\Proba(n) \text{ vertices).}
\end{equation}
This concludes the proof.

\end{appendix}

\begin{acks}[Acknowledgments]
The author used large language model tools for brainstorming, drafting, generating figures, and proofreading. The author takes full responsibility for all mathematical content and any errors.

The author would like to thank Laurent Massoulié for his helpful comments on earlier versions of the paper.
\end{acks}

\end{document}